\documentclass[a4paper, 10pt]{article}
\usepackage{mathrsfs}
\usepackage{mathrsfs}
\usepackage{amsmath,amssymb}
\usepackage{amsfonts}
\usepackage[T1]{fontenc}
\usepackage[latin9]{inputenc}
\usepackage{amsthm}
\usepackage{graphicx}
\usepackage{amsmath}

\usepackage{amsthm}
\usepackage{array}
\usepackage{cases}
\usepackage{enumerate,enumitem,todonotes}
\usepackage{thmtools}
\usepackage{thm-restate}
\usepackage{setspace}  
\usepackage{changes}
\usepackage{float} %ensure the figure is placed where you want [H]

\makeatletter
\def\th@plain{%
  \itshape % body font
}
\makeatother

\makeatletter
\renewenvironment{proof}[1][\proofname]{\par
  \pushQED{\qed}%
  \normalfont \topsep6\p@\@plus6\p@\relax
  \trivlist
  \item[\hskip\labelsep
        \bfseries
    #1\@addpunct{.}]\ignorespaces
}{%
  \popQED\endtrivlist\@endpefalse
}
\makeatother

\newtheorem{theorem}{Theorem}[section]

\newtheorem{conjecture}[theorem]{Conjecture}
\numberwithin{equation}{section}

\newtheorem{thm}{Theorem}[section]
\newtheorem{cor}[thm]{Corollary}
\newtheorem{claim}[thm]{Claim}
\newtheorem{lemma}[thm]{Lemma}
\newtheorem{example}[thm]{Example}

\newtheorem{defn}[thm]{Definition}

\newtheorem{rem}[thm]{Remark}

\numberwithin{equation}{section}

\usepackage[varg]{txfonts}
\usepackage{graphicx, epsfig, subfigure}
\usepackage{enumerate}
\usepackage[square, numbers, sort&compress]{natbib}
\ifx\pdfoutput\undefined
 \usepackage[dvipdfm,%
  pdfstartview=FitH,
  bookmarks=true,%
  bookmarksnumbered=true,
  bookmarksopen=true,
  plainpages=false,%
  pdfpagelabels,%
  colorlinks=true,
  linkcolor=blue,
  citecolor=blue,%
  urlcolor=black,
  pdfborder=001]{hyperref}
  \AtBeginDvi{}  %  GBK  ->  Unicode
\else
 \usepackage[pdftex,%
  pdfstartview=FitH,
  bookmarks=true,%
  bookmarksnumbered=true,
  bookmarksopen=true,
  plainpages=false,%
  pdfpagelabels,%
  colorlinks=true,
  linkcolor=blue,
  citecolor=blue,%
  urlcolor=black,
  pdfborder=001]{hyperref}%
\fi

\usepackage[left]{lineno}
\RequirePackage[normalem]{ulem} %DIF PREAMBLE
\RequirePackage{color}\definecolor{RED}{rgb}{1,0,0}\definecolor{BLUE}{rgb}{0,0,1} %DIF PREAMBLE
\providecommand{\DIFaddbegin}{} %DIF PREAMBLE
\providecommand{\DIFaddend}{} %DIF PREAMBLE
\providecommand{\DIFdelbegin}{} %DIF PREAMBLE
\providecommand{\DIFdelend}{} %DIF PREAMBLE
\providecommand{\DIFaddbeginFL}{} %DIF PREAMBLE
\providecommand{\DIFaddendFL}{} %DIF PREAMBLE
\providecommand{\DIFdelbeginFL}{} %DIF PREAMBLE
\providecommand{\DIFdelendFL}{} %DIF PREAMBLE
\newcommand{\DIFscaledelfig}{0.5}
\RequirePackage{settobox} %DIF PREAMBLE
\RequirePackage{letltxmacro} %DIF PREAMBLE
\newsavebox{\DIFdelgraphicsbox} %DIF PREAMBLE
\newlength{\DIFdelgraphicswidth} %DIF PREAMBLE
\newlength{\DIFdelgraphicsheight} %DIF PREAMBLE
\LetLtxMacro{\DIFOincludegraphics}{\includegraphics} %DIF PREAMBLE
\newcommand{\DIFaddincludegraphics}[2][]{{\color{blue}\fbox{\DIFOincludegraphics[#1]{#2}}}} %DIF PREAMBLE
\newcommand{\DIFdelincludegraphics}[2][]{% %DIF PREAMBLE
\sbox{\DIFdelgraphicsbox}{\DIFOincludegraphics[#1]{#2}}% %DIF PREAMBLE
\settoboxwidth{\DIFdelgraphicswidth}{\DIFdelgraphicsbox} %DIF PREAMBLE
\settoboxtotalheight{\DIFdelgraphicsheight}{\DIFdelgraphicsbox} %DIF PREAMBLE
\scalebox{\DIFscaledelfig}{% %DIF PREAMBLE
\parbox[b]{\DIFdelgraphicswidth}{\usebox{\DIFdelgraphicsbox}\\[-\baselineskip] \rule{\DIFdelgraphicswidth}{0em}}\llap{\resizebox{\DIFdelgraphicswidth}{\DIFdelgraphicsheight}{% %DIF PREAMBLE
\setlength{\unitlength}{\DIFdelgraphicswidth}% %DIF PREAMBLE
\begin{picture}(1,1)% %DIF PREAMBLE
\thicklines\linethickness{2pt} %DIF PREAMBLE
{\color[rgb]{1,0,0}\put(0,0){\framebox(1,1){}}}% %DIF PREAMBLE
{\color[rgb]{1,0,0}\put(0,0){\line( 1,1){1}}}% %DIF PREAMBLE
{\color[rgb]{1,0,0}\put(0,1){\line(1,-1){1}}}% %DIF PREAMBLE
\end{picture}% %DIF PREAMBLE
}\hspace*{3pt}}} %DIF PREAMBLE
} %DIF PREAMBLE
\LetLtxMacro{\DIFOaddbegin}{\DIFaddbegin} %DIF PREAMBLE
\LetLtxMacro{\DIFOaddend}{\DIFaddend} %DIF PREAMBLE
\LetLtxMacro{\DIFOdelbegin}{\DIFdelbegin} %DIF PREAMBLE
\LetLtxMacro{\DIFOdelend}{\DIFdelend} %DIF PREAMBLE
\DeclareRobustCommand{\DIFaddbegin}{\DIFOaddbegin \let\includegraphics\DIFaddincludegraphics} %DIF PREAMBLE
\DeclareRobustCommand{\DIFaddend}{\DIFOaddend \let\includegraphics\DIFOincludegraphics} %DIF PREAMBLE
\DeclareRobustCommand{\DIFdelbegin}{\DIFOdelbegin \let\includegraphics\DIFdelincludegraphics} %DIF PREAMBLE
\DeclareRobustCommand{\DIFdelend}{\DIFOaddend \let\includegraphics\DIFOincludegraphics} %DIF PREAMBLE
\LetLtxMacro{\DIFOaddbeginFL}{\DIFaddbeginFL} %DIF PREAMBLE
\LetLtxMacro{\DIFOaddendFL}{\DIFaddendFL} %DIF PREAMBLE
\LetLtxMacro{\DIFOdelbeginFL}{\DIFdelbeginFL} %DIF PREAMBLE
\LetLtxMacro{\DIFOdelendFL}{\DIFdelendFL} %DIF PREAMBLE
\DeclareRobustCommand{\DIFaddbeginFL}{\DIFOaddbeginFL \let\includegraphics\DIFaddincludegraphics} %DIF PREAMBLE
\DeclareRobustCommand{\DIFaddendFL}{\DIFOaddendFL \let\includegraphics\DIFOincludegraphics} %DIF PREAMBLE
\DeclareRobustCommand{\DIFdelbeginFL}{\DIFOdelbeginFL \let\includegraphics\DIFdelincludegraphics} %DIF PREAMBLE
\DeclareRobustCommand{\DIFdelendFL}{\DIFOaddendFL \let\includegraphics\DIFOincludegraphics} %DIF PREAMBLE
\begin{document}
\title{\LARGE  Packing edge-colorings of subcubic outerplanar graphs

}

%\thanks{Supported by XXXX.}
\author{
Sijin Li\thanks{School of Mathematics and Physics, Xi'an Jiaotong-Liverpool University, Suzhou, Jiangsu Province, 215123, China, \texttt{sijin.li21@student.xjtlu.edu.cn}}
\and
Yifan Li\thanks{School of Mathematics and Physics, Xi'an Jiaotong-Liverpool University, Suzhou, Jiangsu Province, 215123, China, \texttt{yifan.li21@student.xjtlu.edu.cn}}
\and
Xujun Liu\thanks{Department of Foundational Mathematics, School of Mathematics and Physics, Xi'an Jiaotong-Liverpool University, Suzhou, Jiangsu Province, 215123, China, \texttt{xujun.liu@xjtlu.edu.cn}; the research of X. Liu was supported by the National Natural Science Foundation of China under grant No.~12401466 and the Research Development Fund RDF-21-02-066 of Xi'an Jiaotong-Liverpool University.}
}

\date{}

\maketitle

\begin{abstract}
%\hspace{-8mm}

\baselineskip 0.60cm

For a sequence $S = (s_1, s_2, \ldots, s_k)$ of non-decreasing positive integers, an $S$-packing edge-coloring ($S$-coloring) of a graph $G$ is a partition of $E(G)$ into $E_1, E_2, \ldots, E_k$ such that the distance between each pair of distinct edges $e_1,e_2 \in E_i$, $1 \le i \le k$, is at least $s_i + 1$. In particular, a $(1^{\ell},2^k)$-coloring is a partition of $E(G)$ into $\ell$ matchings and $k$ induced matchings, and it can be viewed as intermediate colorings between proper and strong edge-colorings. Hocquard, Lajou, and Lu\v zar conjectured that every subcubic planar graph has a $(1,2^6)$-coloring and a $(1^2,2^3)$-coloring. 

In this paper, we confirm the conjecture of Hocquard, Lajou, and Lu\v zar for subcubic outerplanar graphs by showing every subcubic outerplanar graph has a $(1,2^5)$-coloring and a $(1^2,2^3)$-coloring. Our results are best possible since we found subcubic outerplanar graphs with no $(1,2^4)$-coloring and no $(1^2,2^2)$-coloring respectively.  Furthermore, we explore the question ``What is the largest positive integer $k_1$ and $k_2$ such that every subcubic outerplanar graph is $(1,2^4,k_1)$-colorable and $(1^2,2^2,k_2)$-colorable?''. We prove $3 \le k_1 \le 6$ and $3 \le k_2 \le 4$. We also consider the question ``What is the largest positive integer $k_1'$ and $k_2'$ such that every $2$-connected subcubic outerplanar graph is $(1,2^3,k_1')$-colorable and $(1^2,2^2,k_2')$-colorable?''. We prove $k_1' = 2$ and $3 \le k_2' \le 11$.

\vspace{3mm}\noindent \emph{Keywords}: $S$-packing edge-coloring; outerplanar graph; subcubic graph; matching; induced matching  
\end{abstract}

\baselineskip 0.60cm

\section{Introduction}
For a non-decreasing sequence $S = (s_1, s_2, \ldots, s_k)$ of positive integers, an $S$-packing edge-coloring (abbreviated to $S$-coloring) of a graph $G$ is a partition of $E(G)$ into $E_1, E_2, \ldots, E_k$ such that for each pair of distinct edges $e_1,e_2 \in E_i$, $1 \le i \le k$, the distance between $e_1$ and $e_2$ is at least $s_i + 1$. This notion was first introduced in 2019 by Gastineau and Togni~\cite{GT2}. Its vertex counterpart has also received much attention recently (e.g. see~\cite{BKL, BGT, BKRW, FKL, GHT, GT1, KL}). To simplify the notation, we use exponents to denote repetitions of the same number, e.g., $(1,1,2,2,2,3)$ is denoted by $(1^2,2^3,3)$.

A proper $k$-edge coloring of $G$ is a partition of $E(G)$ into $k$ matchings. The chromatic index of a graph $G$ is the minimum $k$ such that $G$ has a proper $k$-edge coloring. The famous result of Vizing~\cite{V} states that every graph $G$ with $\Delta(G) = \Delta$ either has chromatic index $\Delta$ or $\Delta + 1$. The strong chromatic index $\chi_s'(G)$, introduced by Fouquet and Jolivet~\cite{FJ} in 1983, is the minimum $k$ such that $G$ can be partitioned into $k$ induced matchings. Erd\H{o}s and Ne\v set\v ril~\cite{EN} conjectured that $\chi_s'(G) \le \frac{5}{4} \Delta(G)^2$ when $\Delta(G)$ is even and $\chi_s'(G) \le \frac{5}{4} \Delta(G)^2 - \frac{1}{2} \Delta(G) + \frac{1}{4}$ when $\Delta(G)$ is odd. Using the language of $S$-packing edge-coloring, we know a strong edge-coloring using $k$ colors is equivalent to a $(2^k)$-packing edge-coloring. Many researchers have extended the study of strong edge-coloring (e.g.,~\cite{A, CDYZ, CKKR, FKS, HOV, HQT, HSY, HJK, KLRSWY, LMSS, MR}). In particular,  Andersen~\cite{A} and independently Hor\'ak, Qing, and Trotter~\cite{HQT} proved every subcubic graph has a $(2^{10})$-packing edge-coloring and thus confirming the conjecture of Erd\H{o}s and Ne\v set\v ril for $\Delta(G) = 3$. Kostochka, Li, Ruksasakchai, Santana, Wang, and Yu~\cite{KLRSWY} showed every subcubic planar graph is $(2^9)$-packing edge-colorable. For graphs with maximum degree at most four, the current best upper bound ($21$) was proved by Huang, Santana, and Yu~\cite{HSY}, which is only one color away from the conjectured upper bound of Erd\H{o}s and Ne\v set\v ril when $\Delta(G) = 4$. For large $\Delta(G)$, the current best upper bound, $1.772 \Delta(G)^2$, was proved by Hurley, de Joannis de Verclos, and Kang~\cite{HJK}.

The $(1^{\ell}, 2^k)$-packing edge-colorings (abbreviated to $(1^{\ell}, 2^k)$-colorings) can be viewed as intermediate colorings between proper edge-colorings and strong edge-colorings. Gastineau and Togni~\cite{GT2} asked whether every subcubic graph has a $(1,2^7)$-coloring. Hocquard, Lajou, and Lu\v zar~\cite{HLL} proved every subcubic graph has a $(1,2^8)$-coloring and conjectured the question of Gastineau and Togni is true. For colorings with two $1$-colors, Hocquard, Lajou, and Lu\v zar~\cite{HLL} showed every subcubic graph has a $(1^2,2^5)$-coloring. Furthermore, Hocquard, Lajou, and Lu\v zar~\cite{HLL}, as well as Gastineau and Togni~\cite{GT2}, conjectured that every subcubic graph is $(1^2,2^4)$-colorable. The former conjecture was recently confirmed by Liu, Santana, and Short~\cite{LSS}, who showed every subcubic multigraph is $(1,2^7)$-colorable. The latter conjecture (for connected subcubic graphs with more than 70 vertices) was proved by Liu and Yu~\cite{LY}. Both results are best possible since the graph obtained by subdividing exactly one edge of $K_{3,3}$ is neither $(1,2^6)$-colorable nor $(1^2,2^3)$-colorable. For subcubic planar graphs, Hocquard, Lajou, and Lu\v zar~\cite{HLL} posted the following conjecture.

\begin{conjecture}[Hocquard, Lajou, and Lu\v zar~\cite{HLL}]\label{main-conj}
 Every subcubic planar graph is $(1,2^6)$-packing edge-colorable and $(1^2,2^3)$-packing edge-colorable. 
\end{conjecture}
 
In this paper, we confirm Conjecture~\ref{main-conj} for subcubic outerplanar graphs by showing every subcubic outerplanar graph is $(1,2^5)$-colorable and $(1^2,2^3)$-colorable. Our results are best possible since we found subcubic outerplanar graphs with no $(1,2^4)$-coloring and no $(1^2,2^2)$-coloring respectively. Furthermore, we are interested in how tight our results can be in a different sense. We explore the question ``What is the largest positive integer $k_1$ and $k_2$ such that every subcubic outerplanar graph is $(1,2^4,k_1)$-colorable and $(1^2,2^2,k_2)$-colorable?''. We prove $3 \le k_1 \le 6$ and $3 \le k_2 \le 4$. We also consider the question ``What is the largest positive integer $k_1'$ and $k_2'$ such that every $2$-connected subcubic outerplanar graph is $(1,2^3,k_1')$-colorable and $(1^2,2^2,k_2')$-colorable?''. We prove $k_1' = 2$ and $3 \le k_2' \le 11$.

%\section{Notation and Preliminaries}
We work with outerplane graphs (outerplanar graphs with a fixed drawing where all vertices
are on the outer face). A block in a graph is an inclusion maximal subgraph with no cut vertices. By definition,
each block is either $2$-connected (non-trivial) or a $K_2$ (trivial). Given an outerplane graph G, the weak dual graph, $\tau(G)$, is the graph that has a vertex for every bounded face of the embedding, and an edge for every pair of bounded faces sharing at least one edge. It is well-known that a plane graph is outerplane if and only if its weak dual is a forest. We say a vertex $u$ sees color $c$ if color $c$ is used in an edge incident to $u$. A pattern $[a,b,c]$ is used to denote the use of colors $a,b,c,a,b,c, \ldots$. In our proof, we assume $G$ is connected since otherwise we apply the same argument to each connected component.

We first show in Section 2.1 that $k_1' = 2$. Then we show in Section 2.2 that $3 \le k_1 \le 6$ and hence the conjecture of Hocquard et al (one $1$-color) for subcubic outerplanar graphs is confirmed as a corollary. In Section 3.1, we prove that $3 \le k_2 \le 4$. This result implies the conjecture of Hocquard et al (two $1$-colors) is true for subcubic outerplanar graphs. At last, we show in Section 3.2 that $3 \le k_2' \le 11$ and propose three open problems in Section 4.

%A block in a graph G is pendant if it contains at most one cut vertex of G.

\section{$(1^1,2^k)$-packing edge-colorings of subcubic outerplanar graphs}

In this section, we focus on $(1^1, 2^k)$-packing edge-colorings of subcubic outerplanar graphs. We first show in Section~\ref{12222} that every $2$-connected subcubic outerplanar graph has a $(1,2^4)$-packing edge-coloring. 

\begin{theorem}\label{thm-12222}
Every $2$-connected subcubic outerplanar graph has a $(1,2^4)$-packing edge-coloring.    
\end{theorem}

Our result is the best possible since there are $2$-connected subcubic outerplanar graphs with no $(2^5)$-packing edge-coloring and no $(1,2^3,3)$-packing edge-coloring respectively (Examples~\ref{example_22222} and~\ref{example_12223}). Furthermore, we provide in Example~\ref{no12222} a subcubic outerplanar graph that has no $(1,2^4)$-packing edge-coloring, showing the ``$2$-connected'' condition in Theorem~\ref{thm-12222} also cannot be dropped.

\begin{example}\label{example_22222}
The graph $G_1$ is obtained from a five cycle $u_1u_2u_3u_4u_5$ plus an edge $u_2u_5$ (See Fig.~\ref{example_1} left picture). The graph $G_1$ is $2$-connected, subcubic, outerplanar, and has $6$ edges. Since every pair of edges in $G_1$ has distance at most two, $G_1$ is not $(2^5)$-packing edge-colorable.    
\end{example}

\begin{figure}[ht]
 \vspace{-3mm}
\begin{center}
  \includegraphics[scale=0.6]{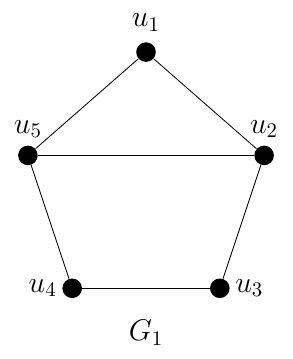} \hspace{5mm}
  \includegraphics[scale=0.65]{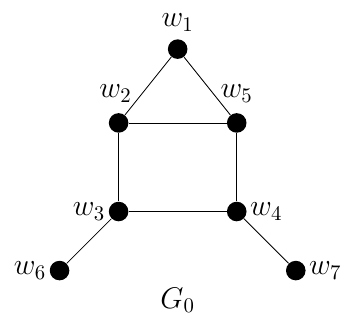}
  \includegraphics[scale=0.6]{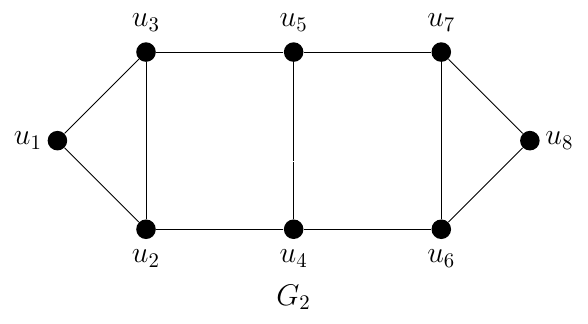}
 \vspace{-3mm}
\caption{$2$-connected subcubic outerplanar graphs with no $(2^5)$-coloring and no $(1,2^3,3)$-coloring.}\label{example_1}
\end{center}
\vspace{-8mm}
\end{figure}

\begin{defn}
Let $G_0$ in Fig.~\ref{example_1} be called the \textbf{house} graph and $G_0-\{w_6,w_7\}$ be called the \textbf{house-without-base} graph.

%{\em house} be the graph induced by $\{w_1, w_2, w_3, w_4, w_5\}$ of $G_0$.
\end{defn}

\begin{example}\label{example_12223}
The graph $G_2$ in Fig.~\ref{example_1} is 2-connected, subcubic, and outerplanar, but not $(1,2^3,3)$-colorable.   
\end{example}

\begin{proof}
We first show $G_2$ has a unique $(1,2^4)$-coloring up to symmetry. 

%Then we show $G_2$ has no $(1,2^3,3)$-coloring.
%Suppose not, i.e., there is a $(1,2^3,3)$-coloring $f: G_2 \to \{1, 2_a, 2_b, 2_c, 3\}$. We show it is impossible.

\begin{claim}\label{unique12222}
The gadget $G_0$ has two $(1,2^4)$-colorings and $G_2$ has a unique $(1,2^4)$-coloring up to symmetry.
\end{claim}

\begin{proof}
Suppose $G_0$ has a $(1,2^4)$-coloring $f$. We first show $f(w_1w_2) \neq 1$. If not, i.e., $f(w_1w_2) = 1$, then it implies $f(w_1w_5), f(w_2w_5), f(w_2w_3) \in \{2_a, 2_b, 2_c, 2_d\}$, say $f(w_1w_5) = 2_a, f(w_2w_5) = 2_b, f(w_2w_3) = 2_c$. Since each of $w_3w_4, w_4w_5, w_4w_7$ is at distance at most two from each of $w_1w_5, w_2w_5, w_2w_3$, the only available colors for $w_3w_4, w_4w_5, $ $w_4w_7$ are $\{1,2_d\}$, which is impossible. By symmetry, $f(w_1w_5) \neq 1$ as well.

Let $G_0'$ be the subgraph induced by vertices $\{w_1, w_2, w_3, w_4, w_5\}$. Since each pair of edges in $G_0'$ has distance at most two, $f$ needs to use $1$ at least twice. Since $f(w_1w_2) \neq 1$ and $f(w_1w_5) \neq 1$, the only two $(1,2^4)$-colorings of $G_0$, up to symmetry, are Type I coloring: $f_1(w_2w_5) = f_1(w_3w_4) = 1$, $f_1(w_1w_2) = f_1(w_4w_7) = 2_a$, $f_1(w_1w_5) = f_1(w_3w_6) = 2_b$, $f_1(w_2w_3) = 2_c$, $f_1(w_4w_5) = 2_d$, and Type II coloring: $f_2(w_2w_3) = f_2(w_4w_5) = 1$, $f_2(w_1w_2) = f_2(w_4w_7) = 2_a$, $f_2(w_1w_5) = f_2(w_3w_6) = 2_b$, $f_2(w_2w_5) = 2_c$, $f_2(w_3w_4) = 2_d$ (see Fig.~\ref{determinedcoloring}).

Since $G_2$ contains two copies of gadget $G_0$, they must be both of Type I. Therefore, the unique $(1,2^4)$-coloring $g$ of $G_2$, up to symmetry, must satisfy $g(u_2u_3) = g(u_4u_5) = g(u_6u_7) = 1$, $g(u_1u_2) = g(u_5u_7) = 2_a$, $g(u_1u_3) = g(u_4u_6) = 2_b$, $g(u_2u_4) = g(u_7u_8) = 2_c$, and $g(u_3u_5) = g(u_6u_8) = 2_d$. 
\end{proof}

Suppose $G_2$ has a $(1,2^3,3)$-coloring $h$. The $(1,2^4)$-coloring $h'$ obtained by changing color $3$ of $h$ to $2_d$ is a valid $(1,2^4)$-coloring of $G_2$. By Claim~\ref{unique12222}, $h'$ must be isomorphic to $g$. However, the two edges colored with $2_d$ are at distance exactly three, and it contradicts the fact that $h$ is a $(1,2^3,3)$-coloring.       
\end{proof}

\begin{rem}\label{twohouses}
The graph $G_2$ (Fig.~\ref{determinedcoloring}) has a unique $(1,2^4)$-coloring, showed in the proof of Claim~\ref{unique12222}. Let the graph $G_2'$ be obtained from $G_2$ by adding a vertex $u$ and an edge $uu_1$. Then $uu_1$ is colored with $1$ in every $(1,2^4)$-coloring of $G_2'$. 
\end{rem}

\begin{figure}[ht]
 \vspace{-5mm}
\begin{center}
  \includegraphics[scale=0.7]{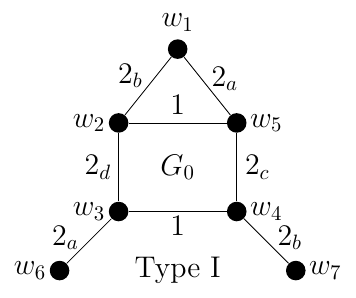} 
  \includegraphics[scale=0.7]{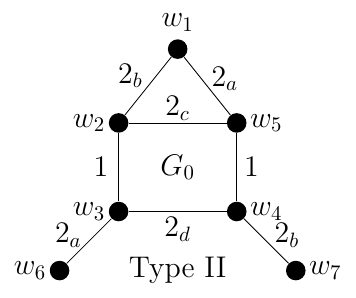}
  \includegraphics[scale=0.6]{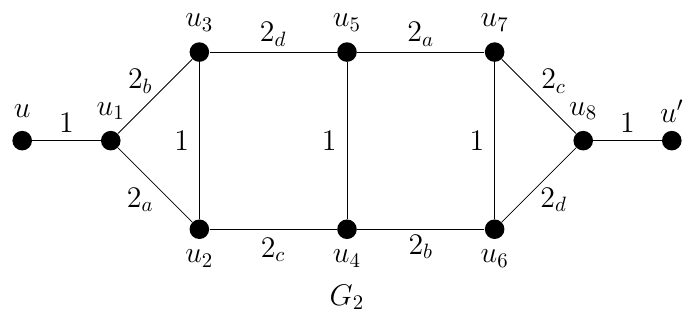}
 \vspace{-3mm}
\caption{Determined $(1,2^4)$-colorings for $G_0$ and $G_2$.}\label{determinedcoloring}
\end{center}
\vspace{-8mm}
\end{figure}

\begin{example}\label{no12222}
The graph $G_3$ (Fig.~\ref{example_3}) is obtained by taking a copy of $G_0$, a copy of $G_2$, and then identifying the vertex $w_6$ of $G_0$ with the vertex $u_1$ of $G_2$. We show that $G_3$ is subcubic, outerplanar, and has no $(1,2^4)$-coloring.
\end{example}

\begin{figure}[ht]
 \vspace{-3mm}
\begin{center}
  \includegraphics[scale=0.6]{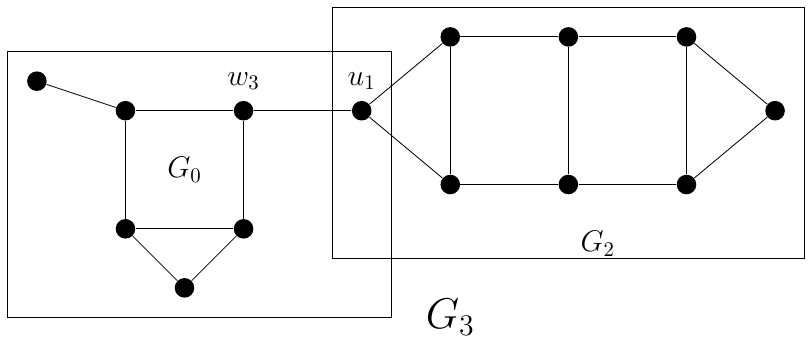} %\hspace{5mm}
   \vspace{-3mm}
\caption{A subcubic outerplanar graph with no $(1,2^4)$-coloring.}\label{example_3}
\end{center}
\vspace{-8mm}
\end{figure}

\begin{proof}
By Claim~\ref{unique12222}, the edge $w_3u_1$ of $G_0$ must use a color from $\{2_a, 2_b, 2_c, 2_d\}$. However, by Remark~\ref{twohouses}, the edge $w_3u_1$ must use color $1$. This is a contradiction.   
\end{proof}

%By Claim~\ref{unique12222}, $G_0$ has two different $(1,2^4)$-colorings $f_1,f_2$ and $G_2$ has a unique $(1,2^4)$-coloring $g$ up to symmetry. By the unique coloring $g$ of $G_2$, the edge $u_1u_3$ of $G_3$ must be colored with $1$. We notice $f_1(w_3w_6)$, $f_1(w_4w_7)$, $f_2(w_3w_6), f_2(w_4w_7) \in \{2_a, 2_b, 2_c, 2_d\}$. This is a contradiction. 

%%%%%%%%%%%%%%%%%%%%%%%%

Inspired by Conjecture~\ref{main-conj}, we explore the question ``what is the largest positive integer $k_1$ such that every subcubic outerplanar graph is $(1,2^4,k_1)$-colorable?''. We prove in Section~\ref{12222k} that $k_1 \ge 3$ and $k_1 \le 6$. 

\begin{theorem}
Every subcubic outerplanar graph is $(1,2^4,3)$-packing edge-colorable.  
\end{theorem}

As a corollary of our result, every subcubic outerplanar graph is $(1,2^5)$-colorable, and thus we confirmed the conjecture (one $1$-color) of Hocquard et al~\cite{HLL} for subcubic outerplanar graphs. 

\begin{cor}\label{cor-1}
Every subcubic outerplanar graph is $(1,2^5)$-packing edge-colorable.   
\end{cor}

Corollary~\ref{cor-1} is also sharp in the sense that there are subcubic outerplanar graphs with no $(1,2^4)$-coloring (Example~\ref{no12222}). Note that Corollary~\ref{cor-1} can also be implied by a result of Hocquard, Ochem, and Valicov~\cite{HOV}, who showed every subcubic outerplanar graph is $(2^6)$-colorable.

\subsection{$(1,2^4)$-packing edge-coloring of $2$-connected subcubic outerplanar graphs}\label{12222}

%%%%%%%%%%%%%%%%%%%%%%%%%%%%%%%%%%% pendant is triangle %%%%%%%%%%%%%%%%%%%%%%%%%%%%%%%%%%%

In this subsection, we prove Theorem~\ref{12222-extended}, where we require all the ``house graph'' must be using Type I coloring (shown in Fig.~\ref{determinedcoloring}). Theorem~\ref{12222} is true as a corollary of Theorem~\ref{12222-extended}.

\begin{theorem}\label{12222-extended}
Every $2$-connected subcubic outerplanar graph has a $(1,2^4)$-packing edge-coloring such that all the ``house graph'' is colored with Type I coloring (we call such a coloring {\em good coloring} in this subsection).   
\end{theorem}

\textbf{Proof of Theorem~\ref{12222-extended}:} Suppose not, i.e., there exist $2$-connected subcubic outerplanar graphs that have no good coloring. We take the graph $G$ with minimum $|V(G)|+|E(G)|$. We may assume the weak dual of $G$ has diameter at least two, since otherwise $G$ itself is a cycle and we color the edges with the pattern $[1,2_a,2_b]$ (we can use colors $2_c, 2_d$ at the end to avoid conflicts between $2_a$s and $2_b$s).

\begin{lemma}\label{lem:pendant_is_triangle}
    Every pendant face of $G$ is a triangle.
\end{lemma}

\begin{proof} %Need to add graph to show the coloring of big cycle $C$
Suppose $F_1$ is a pendant face with the two $3$-vertices $v_1$ and $v_k$. Let the cycle $C=v_1v_2\cdots v_k$ be the boundary of $F_1$. Let $N(v_1) = \{u_1, v_2, v_k\}$ and $N(v_k) = \{u_k, v_1, v_{k-1}\}$. If $k\geq 4$, then we delete $v_2, \ldots, v_{k-1}$ and add a vertex $v$ with edges $vv_1$ and $vv_k$ to obtain a new graph $G'$. The graph $G'$ is a 2-connected subcubic outerplanar graph with $\vert V(G')\vert+\vert E(G')\vert < |V(G)|+ |E(G)|$. By the minimality of $G$, $G'$ has a good coloring $f$. We split the cases according to the number of $1$-color used on $u_1v_1, v_1v_k, v_ku_k$. We extend $f$ to $G$.

\textbf{Case 1:} The edges $u_1v_1$, $v_1v_k$, $v_ku_k$ have two $1$-colors. By symmetry, we may assume $f(u_1v_1) = f(u_kv_k) = 1, f(v_1v_k) = 2_a, f(vv_1) = 2_b$, and $f(vv_k) = 2_c$. Then for the pendant face $F_1$ of $G$, we use $2_b$ to color $v_1v_2$ and $2_c$ to color $v_{k-1}v_k$, and the pattern $[1, 2_d, 1, 2_b]$ to color the edges from $v_2v_3$ to $v_{k-2}v_{k-1}$.

%%%%%%%%%%%%%%%% One 1 %%%%%%%%%%%%%%%%
\textbf{Case 2:} The edges $u_1v_1$, $v_1v_k$, $v_ku_k$ have one $1$-color. By symmetry, there are two subcases to consider.

\textbf{Case 2.1:} $f(v_1v_k) = 1$.  We may assume $f(u_1v_1) = 2_a, f(u_kv_k) = 2_b, f(vv_1) = 2_c, f(vv_k) = 2_d$. We use $2_c$ to color $v_1v_2$ and $2_d$ to color $v_{k-1}v_k$, the pattern $[1, 2_a, 1, 2_c]$ to color the edges from $v_2v_3$ to $v_{k-2}v_{k-1}$. %Then the pendant face of $G$ is colored by $(1, 2, 2, 2, 2)$.

\textbf{Case 2.2:} $f(u_1v_1) = 1$. We may assume $f(v_1v_k) = 2_a, f(v_ku_k) = 2_b$, and $f(vv_1) = 2_c$. We use $2_c$ to color $v_1v_2$ and $1$ to color $v_{k-1}v_k$. We use the pattern $[2_d, 1, 2_b, 1]$ to color the edges from $v_{k-2}v_{k-1}$ to $v_2v_3$. %Then the pendant face of $G$ is colored by $(1, 2, 2, 2, 2)$.

\textbf{Case 3:} The edges $u_1v_1$, $v_1v_k$, $v_ku_k$ have no $1$-color. We may assume $f(u_1v_1) = 2_a, f(v_1v_k) = 2_b, f(u_kv_k) = 2_c, f(vv_k) = 2_d, f(vv_1) = 1$. We can switch the colors of $vv_1$ and $v_1v_k$, which was already solved in \textbf{Case 2.1}.
\end{proof}

%%%%%%%%%%%%%%%%%%%%%%%%%%%%%%%%%%% 2-pendant is a house %%%%%%%%%%%%%%%%%%%%%%%%%%%%%%%%%%%
We may assume $G$ has at least three faces since otherwise, by Lemma~\ref{lem:pendant_is_triangle}, it must be two triangles sharing an edge. In this case, we have five colors to color five edges. We show next a lemma about the ``most outside two layers'' of $G$.
\begin{lemma}\label{lem12222:outside_two_layer}
Let $F_1$ be a pendant face which is corresponding to a leaf in a longest path $P$ in the weak dual of $G$. Let $F_2$ be the unique face adjacent to $F_1$. Then $F_2$ must be a $4$-face.
\end{lemma}       

\begin{proof}
Suppose not, i.e., $F_2$ is a $k$-face with $k \ge 5$. Since $G$ has at least three faces, $P$ has at least $3$ vertices. Let $F_3$ be the face which corresponds to the other neighbour of $F_2$ in $P$. Let the boundary cycle of $F_2$ be $v_1v_2 \cdots v_k$, where $v_1, v_k \in V(F_2) \cap V(F_3)$. We delete $v_2, \ldots, v_{k-1}$ and all pendant faces adjacent to it (which must be triangles). Then we add $w, w_2, w_3$ with edges $v_1w_2, w_2w_3, w_3v_k, ww_2, ww_3$ to obtain $G'$, which is a 2-connected subcubic outerplanar graph with $\vert V(G')\vert+\vert E(G')\vert < |V(G)| + |E(G)|$ . By the minimality of $G$, $G'$ has a good coloring $f$. We extend $f$ to $G$. 

Let $N_G(v_1) = \{u_1, v_2, v_k\}$ and $N_G(v_k) = \{v_1, v_{k-1}, u_k\}$. Then we know $f(u_1v_1) = f(ww_3) = 2_a$, $f(u_kv_k) = f(ww_2) = 2_b$, $f(w_3v_k) = 2_c$, $f(v_1w_2) = 2_d$, and $f(v_1v_k) = f(w_2w_3) = 1$. We first color $v_1v_2, v_{k-1}v_k$ with $2_d,2_c$, then use the pattern $[1, 2_a, 1, 2_d]$ to color the edges $v_2v_3$ through $v_{k-2}v_{k-1}$ in the boundary of $F_2$. If there is a $3$-face on top of $v_iv_{i+1}$, then we call the third vertex $v_i'$, where $2 \le i \le k-2$. We split the proof into two cases based on the largest subscript $j$ of $v_j'$.

\textbf{Case 1:} $j \le k-3$. We color $v_{j}v_{j}', v_j'v_{j+1}$ with $2_c, 2_b$, and the other uncolored edges of $3$-faces $v_iv_i', v_{i}'v_{i+1}$ with $2_c, 2_b$. 

\textbf{Case 2:} $j = k-2$. If $f(v_{k-2}v_{k-1}) = 1$, then color $v_{k-2}v_{k-2}', v_{k-2}'v_{k-1}$ with $2_b,x$, $x \in \{2_a, 2_d\} - f(v_{k-3}v_{k-2})$. Otherwise, we color $v_{k-2}v_{k-2}',v_{k-2}'v_{k-1}$ with $2_b,1$. The other uncolored edges of $3$-faces $v_iv_i', v_{i}'v_{i+1}$ are colored with $2_b, 2_c$.
\end{proof}

Now, we complete the proof of Theorem~\ref{12222-extended}. By Lemma~\ref{lem12222:outside_two_layer}, if $P$ only has three vertices, then $G$ is a special house with $w_6 = w_7$ in Fig.~\ref{determinedcoloring}. Then we can color $G$ with the Type I coloring. Therefore, we may assume $P$ has at least $4$ vertices. Let $F_1$ be a leaf of $P$, and $P = F_1F_2F_3F_4 \ldots$. Let the boundary cycle of $F_3$ be $v_1 \ldots v_k$, $k \ge 4$, with $v_1,v_k \in V(F_3) \cap V(F_4)$. We delete $\{v_2, \ldots, v_{k-1}\}$ together with all the ``houses'' and $3$-faces attached to $F_3$ (except $F_4$), and add vertices $w, w_1, w_2$ with edges $v_1w_1, w_1w_2, w_2v_k, ww_1, ww_2$ to obtain a graph $G'$ with $|V(G')| + |E(G')|<|V(G)|+|E(G)|$. By the minimality of $G$, $G'$ has a good coloring $f$. We extend $f$ to $G$. 

In our proof, we may assume only houses can be attached to $F_3$ (except $F_4$), since type I coloring guarantees that the two edges $w_2w_3$ and $w_4w_5$ in $G_0$ (of Fig.~\ref{determinedcoloring}) are colored with different colors in $\{2_a, 2_b, 2_c, 2_d\}$. We also assume there are no adjacent $2$-vertices on the boundary cycle of $F_3$, since otherwise we add a house on top of them. If there is a house attached to $F_3$, where the house and $F_3$ share vertices $v_i,v_{i+1}$, then the three vertices of the house not in $F_3$ are named $\{v_i',v_{i+1}', v_i''\}$ with $v_iv_i', v_{i+1}v_{i+1}', v_i'v_{i+1}'', v_i'v_i'', v_i''v_{i+1}' \in E(G)$. We split the proof into three cases depending on $k$.

\textbf{Case 1:} $k \equiv 2 \mod{3}$. We color the boundary face of $F_3$ starting from $v_1v_2$ using the pattern $[2_d,2_b,2_a]$ until the edge $v_{k-2}v_{k-1}$ and color $v_{k-1}v_k$ with $2_c$. Since $k \equiv 2 \mod{3}$, $v_{k-2}v_{k-1}$ is colored with $2_a$. For the houses attached to $F_3$ on $v_i,v_{i+1}$, where $2 \le i \le k-3$, we color $v_{i+1}v_{i+1}'$ with $f(v_iv_{i+1})$ and recolor $v_iv_{i+1}$ with $1$, then we color $v_iv_i', v_i'v_{i+1}', v_i'v_i'', v_i''v_{i+1}'$ with $2_c,1, f(v_{i+1}v_{i+2}), f(v_{i-1}v_i)$. In case $v_{k-1}$ is a $3$-vertex, then we recolor $v_{k-2}v_{k-1}$ with $1$, and color $v_{k-1}v_{k-1}', v_{k-2}v_{k-2}', v_{k-2}'v_{k-1}',$
$v_{k-2}'v_{k-2}'', v_{k-2}''v_{k-1}'$ with $2_d, 2_a, 1, 2_c, 2_b$.

\textbf{Case 2:} $k \equiv 1 \mod{3}$. We color the boundary face of $F_3$ starting from $v_1v_2$ using the pattern $[2_d,2_b,2_a]$ until the edge $v_{k-2}v_{k-1}$ and color $v_{k-1}v_k$ with $2_c$. Since $k \equiv 1 \mod{3}$, $v_{k-2}v_{k-1}$ is colored with $2_b$. If $v_{k-1}$ is a $3$-vertex, we recolor $v_{k-2}v_{k-1}$ with $1$ and color $v_{k-2}v_{k-2}', v_{k-1}v_{k-1}', v_{k-1}'v_{k-2}', v_{k-2}'v_{k-2}'', v_{k-2}''v_{k-1}'$ with $2_b, 2_a, 1, 2_c, 2_d$. If $v_{k-1}$ is a $2$-vertex, then $v_{k-2}$ is a $3$-vertex, we recolor $v_{k-3}v_{k-2}, v_{k-2}v_{k-1}$ with $1,2_d$ and color $v_{k-3}v_{k-3}', v_{k-2}v_{k-2}', v_{k-2}'v_{k-3}', v_{k-3}'v_{k-3}'', v_{k-3}''v_{k-2}'$ with $2_c, 2_b, 1, 2_d, 2_a$. For the other houses attached to $F_3$ on $v_i,v_{i+1}$, where $2 \le i \le k-4$, we color $v_{i+1}v_{i+1}'$ with $f(v_iv_{i+1})$ and recolor $v_iv_{i+1}$ with $1$, then we color $v_iv_i', v_i'v_{i+1}', v_i'v_i'', v_i''v_{i+1}'$ with $2_c,1, f(v_{i+1}v_{i+2}), f(v_{i-1}v_i)$.

\textbf{Case 3:} $k \equiv 0 \mod{3}$. We color the boundary face of $F_3$ starting from $v_1v_2$ using the pattern $[2_d,2_b,2_a]$ until the edge $v_{k-2}v_{k-1}$ and color $v_{k-1}v_k$ with $2_c$. If $v_{k-1}$ is a $2$-vertex, then for the houses attached to $F_3$ on $v_i,v_{i+1}$, where $2 \le i \le k-3$, we color $v_{i+1}v_{i+1}'$ with $f(v_iv_{i+1})$ and recolor $v_iv_{i+1}$ with $1$, then we color $v_iv_i', v_i'v_{i+1}', v_i'v_i'', v_i''v_{i+1}'$ with $2_c,1, f(v_{i+1}v_{i+2}), f(v_{i-1}v_i)$. Thus, we may assume $v_{k-1}$ is a $3$-vertex. We define $j$ as the largest index $2 \le i \le k-3$ such that $v_j$ is a $2$-vertex (if such a vertex does not exist, then $j=1$). If $j = k-3$, then $k \ge 9$ and $v_{k-4}$ is a $3$-vertex. We recolor $v_{k-2}v_{k-1}, v_{k-5}v_{k-4}, v_{k-4}v_{k-3}$ with $1,1,2_d$, and color $v_{k-5}v_{k-5}', v_{k-4}v_{k-4}', v_{k-5}'v_{k-4}', v_{k-5}'v_{k-5}'', v_{k-5}''v_{k-4}'$ with $2_c, 2_b, 1, 2_d, 2_a$, and color $v_{k-2}v_{k-2}', v_{k-1}v_{k-1}', v_{k-2}'v_{k-1}', v_{k-2}'v_{k-2}'', v_{k-2}''v_{k-1}'$ with $2_b, 2_d, 1, 2_c, 2_a$. If $j \le k-4$, then it implies $j \le k-5$, we color $v_{j+1}v_{j+1}', v_{j+2}v_{j+2}', v_{j+1}'v_{j+2}', v_{j+1}'v_{j+1}'', v_{j+1}''v_{j+2}'$ with $2_c, f(v_{j+2}v_{j+3}), 1, f(v_{j+1}v_{j+2}), f(v_jv_{j+1})$, recolor the edges $v_{j+2}v_{j+3}, \ldots, v_{k-2}v_{k-1}$ with the colors $f(v_{j+1}v_{j+2}), \ldots, f(v_{k-3}v_{k-2})$, and recolor $v_{j+1}v_{j+2}$ with $1$. Let the recolored coloring be called $f_1$. For all the other houses attached to $F_3$ on $v_i,v_{i+1}$, where $2 \le i \le k-3$ ($i \neq j+1$ when $j \le k-5$; $i \neq j+1$ and $i \neq j-2$ when $j = k-3$), we color $v_{i+1}v_{i+1}'$ with $f_1(v_iv_{i+1})$ and recolor $v_iv_{i+1}$ with $1$, then we color $v_iv_i', v_i'v_{i+1}', v_i'v_i'', v_i''v_{i+1}'$ with $2_c,1, f_1(v_{i+1}v_{i+2}), f_1(v_{i-1}v_i)$. \hfill \qed

%ADD A PICTURE!

\subsection{$(1,2^4,k_1)$-packing edge-coloring of subcubic outerplanar graphs}\label{12222k}

\begin{example}
The gadget $G_0$ and graph $G$ are shown in Fig.~\ref{example_4}. Each of $G_0'$ and $G_0''$ is isomorphic to $G_0$. The vertices $v',v''$ in $G_0',G_0''$ correspond to $v$ in $G_0$. We show $G$ has no $(1,2^4,7)$-packing edge-coloring.
\end{example}

\begin{figure}[ht]
\vspace{-11mm}
\begin{center}
  \includegraphics[scale=0.65]{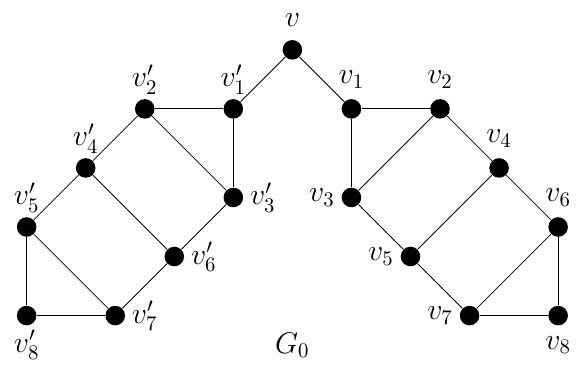} \hspace{10mm}
  \includegraphics[scale=0.4]{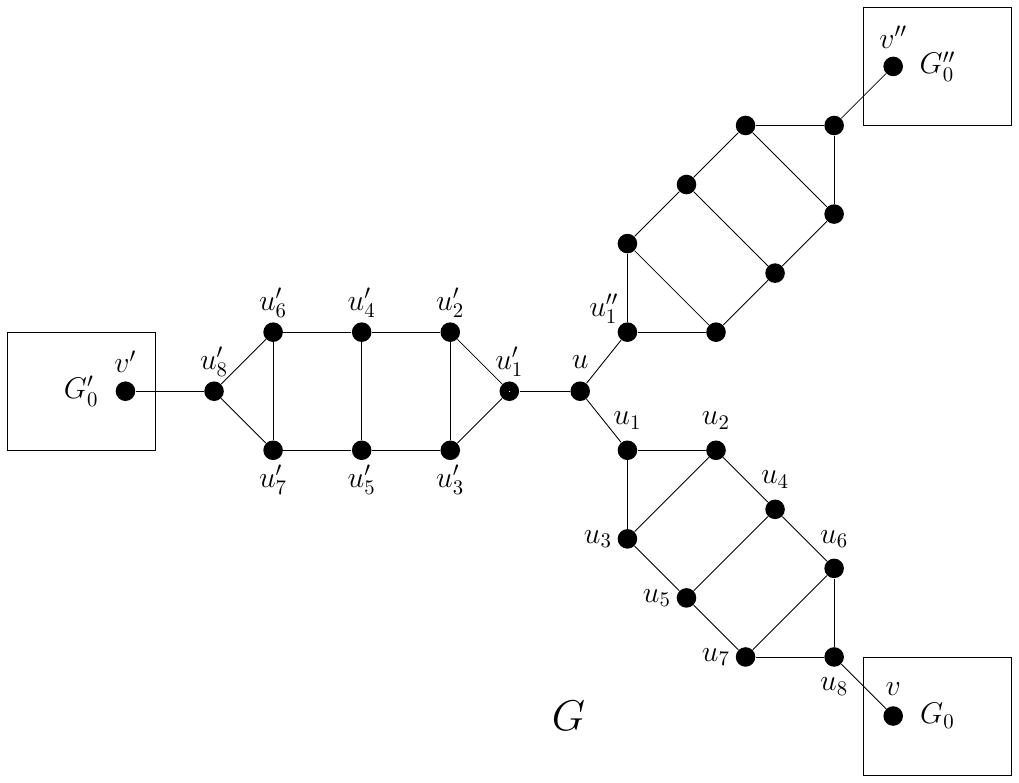} %\hspace{10mm}
   \vspace{-3mm}
\caption{A subcubic outerplanar graph that is not $(1,2^4,7)$-packing edge-colorable.}\label{example_4}
\end{center}
\vspace{-8mm}
\end{figure}

\begin{proof}
Suppose $G$ has a $(1,2^4,7)$-packing edge-coloring $f$. We first consider the case when one of $uu_1, uu_1', uu_1''$ is colored with $7$, say $f(uu_1'') = 7$. Then none of the edges in the graph $G_1$ induced by $\{u_1, \ldots, u_8\}$ and the graph $G_1'$ induced by $\{u_1', \ldots, u_8'\}$ can be colored with $7$. By Remark~\ref{twohouses}, both $uu_1$ and $uu_1'$ must be colored with $1$, which is a contradiction. Thus, we know color $7$ is not used on $uu_1, uu_1', uu_1''$.

Since $uu_1, uu_1', uu_1''$ have the same endpoint $u$, at least two of them must use a color in $\{2_a, 2_b, 2_c, 2_d\}$, say $f(uu_1) = 2_a$ and $f(uu_1') = 2_b$. By Remark~\ref{twohouses}, color $7$ must be used on both $G_1$ and $G_1'$. Due to the distance restriction, color $7$ is used on either the triangle $u_6u_7u_8$ or $u_6'u_7'u_8'$, say color $7$ is used on the triangle $u_6u_7u_8$. Then none of $vu_8, vv_1, vv_1'$, and the edges of $G_0$ can be colored with $7$. However, at least one of $vv_1$ and $vv_1'$ must be colored with a color in $\{2_a, 2_b, 2_c, 2_d\}$. This is a contradiction with Remark~\ref{twohouses}.
\end{proof}

%Define the distance between a vertex and an edge. Define what do we mean by saying sees

% The name V_1 need to be changed to V_0'

Since every subcubic outerplanar graph $G$ is a subgraph of another subgraph $\hat{G}$ such that $\hat{G}$ has no $2$-vertices (by adding leaf edges to each $2$-vertices in $G$). We prove the following theorem on $(1,2^4,3)$-coloring of subcubic outerplanar graphs with no $2$-vertices. Let the graph in Fig.~\ref{house-with-chimney-1} be called the {\em ``house-with-chimney'' graph (HC graph)}. We define three types of coloring for the HC graph (See Fig.~\ref{house-with-chimney-1} for Types I, II, and III). If every HC graph is colored using Types I, II, and III unless the ``house-without-base'' graph is itself a block, and the only neighbor of each vertex of degree one does not see color $3$, then we call such a coloring {\em good coloring} in this subsection.

\begin{figure}[ht]
\begin{center}
 \vspace{-2mm}
 \includegraphics[scale=0.6]{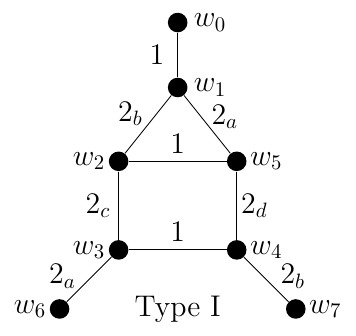} \hspace{8mm}
  \includegraphics[scale=0.6]{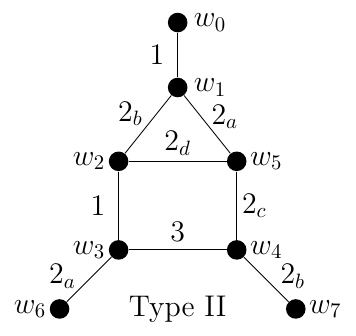} \hspace{8mm}
   \includegraphics[scale=0.6]{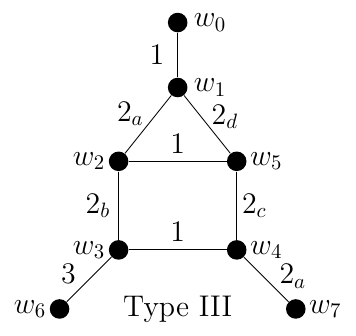} %\hspace{5mm}  
   \vspace{-3mm}
\caption{Type I, II, III colorings of the ``house-with-chimney'' graph.}\label{house-with-chimney-1}
\end{center}
\vspace{-8mm}
\end{figure}

%ADD A PICTURE ON TYPE I II III

\begin{theorem}\label{122223-main-theorem}
Every subcubic outerplanar graph with no $2$-vertices has a good coloring.    
\end{theorem}

\textbf{Proof of Theorem~\ref{122223-main-theorem}:} Suppose not, i.e., there is a subcubic outerplanar graph with no $2$-vertices and it has no good coloring. We take the graph $G$ with minimum $|V(G)| + |E(G)|$. We may assume there is no tree of at least two edges hanging on a cycle in $G$, i.e., if there is a tree hanging on a cycle, then it must be an edge. To see this, let $T$ be a tree hanging on a cycle $C$ in $G$, where $u \in V(C) \cap V(T)$, and we assume $T$ has at least two edges (in fact, at least three edges). Let $P_1$ be a longest path in $T$ ended at $u$, say $P = u \ldots u_1$. Let the only neighbour of $u_1$ be $u_2$. Since there are no $2$-vertices, $u_2$ has another leaf neighbour $u_1'$. Since $P_1$ has length at least $2$, $u_2 \neq u$. We delete $u_1,u_1'$ from $G$ to obtain a graph $G'$, which is a subcubic outerplanar graph with no $2$-vertices and satisfies $|V(G')|+|E(G')|<|V(G)|+|E(G)|$. By the minimality of $G$, $G'$ has a good coloring $f$. Since there are at most $3$ edges within distance two of $u_1u_2, u_1'u_2$ in $G$ and we have $5$ non-restricted colors $1,2_a,2_b,2_c,2_d$ to use at $u_1u_2,u_1'u_2$, we can extend $f$ to $G$. By the same proof, we may assume that $G$ has at least one cycle. In other words, $G$ has at least one non-trivial block.

We pick a non-trivial leaf block $B_1$. Let $u_0$ be the vertex of $B_1$ connecting the other blocks in $G$ and the face containing $u_0$ be $F_0$ (if there is only one non-trivial block $B_1$, then we pick any leaf edge $v_0u_0$ in a cycle of $B_1$, where $v_0$ is the $1$-vertex). We may assume $B_1$ has at least two faces. Otherwise, let $u_0u_1 \ldots u_k$ be the boundary cycle of $F_0$, where $k \ge 2$, and $N(u_0) = \{u_1, u_k, v_0\}$ and $N(v_0) = \{u_0, v_0', v_0''\}$. Since $G$ has no $2$-vertices, there is a leaf edge $u_iu_i'$ hanging on each of $u_i$, where $1 \le i \le k$. We delete all $u_i,u_i'$, $1 \le i \le k$, to obtain $G'$. By the minimality of $G$, $G'$ has a good coloring $f$. Then $|\{f(v_0u_0), f(v_0v_0'), f(v_0v_0'')\} \cap \{2_a, 2_b, 2_c, 2_d\}| \le 2$, since if it is $3$ then we recolor $u_0v_0$ to $1$. %There are two cases depending on the color used on $u_0v_0$.

\textbf{Case 1:} $f(u_0v_0) = 1$. We color $u_0u_1, u_0u_k$ with two colors in $\{2_a, 2_b, 2_c, 2_d\}$, say $2_a,2_b$. Color $u_iu_i'$ with $1$, $1 \le i \le k$. If $k \equiv 0 \text{ or } 2 \text{ } (mod \text{ } 3)$, then use the pattern $[2_c, 2_d, 2_a]$ to color the edges $u_1u_2$ until $u_{k-1}u_{k}$. If $k = 4$, then color $u_1u_2, u_2u_3, u_3u_4$ with $2_c, 1, 2_d$ and recolor $u_1u_1', u_2u_2', u_3u_3', u_4u_4'$ with $1,2_b,2_a,1$. If $k \equiv 1 \text{ } (mod \text{ } 3)$ and $k \neq 4$, then color $u_1u_2,u_2u_3$ with $2_c, 2_b$ and use the pattern $[2_a, 2_c, 2_d]$ to color the edges $u_3u_4$ until $u_{k-1}u_{k}$.

\textbf{Case 2:} $f(u_0v_0) \in \{2_a,2_b, 2_c, 2_d\}$, say $f(u_0v_0) = 2_c$ and we can color $u_0u_k$ with $2_b$. Color $u_iu_i'$ with $1$, $2 \le i \le k$, $u_0u_1$ with $1$, and $u_1u_1'$ with $2_d$. If $k \equiv 1 \text{ or } 2 \text{ } (mod \text{ } 3)$, then use the pattern $[2_a, 2_c, 2_d]$ to color the edges $u_1u_2$ until $u_{k-1}u_{k}$. If $k = 3$, then color $u_1u_2, u_2u_3$ with $2_a, 1$ and recolor $u_2u_2', u_3u_3'$ with $2_c,2_d$. If $k \equiv 0 \text{ } (mod \text{ } 3)$ and $k \neq 3$, then color $u_1u_2,u_2u_3$ with $2_a, 2_b$ and use the pattern $[2_c, 2_d, 2_a]$ to color the edges $u_3u_4$ until $u_{k-1}u_{k}$.

We first show every pendant face (except $F_0$) in $B_1$ is a $3$-face or a $4$-face. 

\begin{lemma}\label{122223-outside-1}
Let $F_1$ be a pendant face in $B_1$ and $F_1 \neq F_0$. Then $F_1$ is a $3$-face or a $4$-face.   
\end{lemma}

\begin{proof}
Suppose $F_1$ is a $k$-face, where $k \ge 5$. Let the only face adjacent to $F_1$ be $F_2$ and the boundary cycle of $F_1$ be $v_1 \ldots v_k$ with $V(F_1) \cap V(F_2) = \{v_1, v_k\}$. Since there are no $2$-vertices, each $v_i$ has a leaf neighbour $v_i'$, $2 \le i \le k-1$. Let $N(v_1) = \{u_1, v_2, v_k\}$ and $N(v_k) = \{u_k,v_1, v_{k-1}\}$. We delete all $v_i,v_i'$, $2 \le i \le k-1$, and add a vertex $v,v'$ with edges $vv_1, vv_k, vv'$ to obtain a graph $G'$. By the minimality of $G$, $G'$ has a good coloring $f$. 
 
\textbf{Case 1:} Color $3$ is not used on $u_1v_1, v_1v_k, v_ku_k$.

\textbf{Case 1.1:} Color $1$ is used twice on $u_1v_1, v_1v_k, v_ku_k$. By the definition of good coloring, we may assume $f(u_1v_1) = f(v_ku_k) = 1$, $f(v_1v_k) = 2_a$, $f(vv_1) = 2_b$, $f(vv_k) = 2_c$, $f(vv') = 1$. Then we color $v_1v_2, v_{k-1}v_k$ with $2_b, 2_c$. If $k \equiv 2 \text{ }(\text{mod } 3)$, then we color $v_{k-2}v_{k-1}$ with $1$ and $v_2v_3$ until $v_{k-3}v_{k-2}$ using the pattern $[2_d,2_a,2_b]$. Furthermore, we color $v_{k-1}v_{k-1}', v_{k-2}v_{k-2}'$ with $2_b,2_a$ and all $v_iv_i'$ with $1$, $2 \le i \le k-3$. If $k \equiv 0 \text{ or } 1 \text{ }(\text{mod } 3)$, then we color the edges $v_2v_3$ until $v_{k-2}v_{k-1}$ using the pattern $[2_d,2_a,2_b]$ and all $v_iv_i'$ with $1$, $2 \le i \le k-1$.

\textbf{Case 1.2:} Color $1$ is used once on $u_1v_1, v_1v_k, v_ku_k$. There are two subcases.

\textbf{Case 1.2.1:} $f(u_1v_1) = f(vv_k) = 1$, $f(v_1v_k) = 2_a$, $f(v_ku_k) = 2_b$, $f(vv_1) = 2_c$, $f(vv') = 2_d$. Then we color $v_1v_2, v_{k-1}v_k, v_{k-1}v_{k-1}'$ with $2_c, 1, 2_c$ and $v_iv_i'$ with $1$, where $2 \le i \le k-2$. If $k \equiv 0 \text{ }(\text{mod } 3)$, then we color $v_{k-2}v_{k-1}$ until $v_2v_3$ using the pattern $[2_d,2_a,2_b]$. Otherwise, we color $v_{k-2}v_{k-1}$ until $v_2v_3$ using the pattern $[2_d,2_b,2_a]$. 

\textbf{Case 1.2.2:} $f(v_1v_k) = f(vv') = 1$, $f(u_1v_1) = 2_a$, $f(v_ku_k) = 2_b$, $f(vv_1) = 2_c$, $f(vv_k) = 2_d$. Then we color $v_1v_2, v_{k-1}v_k$ with $2_c, 2_d$. If $k \equiv 1 \text{ }(\text{mod } 3)$, then we color $v_{k-2}v_{k-1}$ with $1$ and $v_2v_3$ until $v_{k-3}v_{k-2}$ using the pattern $[2_b,2_a,2_c]$. Furthermore, we color $v_{k-1}v_{k-1}', v_{k-2}v_{k-2}'$ with $2_a,2_b$ and all $v_iv_i'$ with $1$, $2 \le i \le k-3$. If $k \equiv 0 \text{ or } 2 \text{ }(\text{mod } 3)$, then we color the edges $v_2v_3$ until $v_{k-2}v_{k-1}$ using the pattern $[2_b,2_a,2_c]$ and all $v_iv_i'$ with $1$, $2 \le i \le k-1$.

\textbf{Case 1.3:} Color $1$ is not used on $u_1v_1, v_1v_k, v_ku_k$. By the definition of good coloring, We may assume $f(u_1v_1)=2_a$, $f(v_1v_k) = 2_b$, $f(v_ku_k) = 2_c$. However, $vv', vv_1, vv_k$ can only be colored by $\{1, 2_d\}$, which is impossible.

\textbf{Case 2:} Color $3$ is used on $u_1v_1, v_1v_k, v_ku_k$.

\textbf{Case 2.1:} Color $1$ is used twice on $u_1v_1, v_1v_k, v_ku_k$. By the definition of good coloring, we may assume $f(u_1v_1) = f(v_ku_k) = 1$, $f(v_1v_k) = 3$, $f(vv_1) = 2_a$, $f(vv_k) = 2_b$. Then we color $v_1v_2, v_{k-1}v_k$ with $2_a, 2_b$ and the edges $v_2v_3$ until $v_{k-2}v_{k-1}$ using the pattern $[2_c, 2_d, 2_a]$. Furthermore, we color $v_iv_i'$ with $1$, where $2 \le i \le k-1$.

\textbf{Case 2.2:} Color $1$ is used once on $u_1v_1, v_1v_k, v_ku_k$. There are three subcases.

\textbf{Case 2.2.1:} $f(u_1v_1) = f(vv_k) = 1$, $f(v_1v_k) = 3$, $f(v_ku_k) = 2_a$, $f(vv_1) = 2_b$, $f(vv') = 2_c$. We color $v_1v_2, v_{k-1}v_k$ with $2_b, 1$ and $v_{k-2}v_{k-1}$ until $v_2v_3$ using the pattern $[2_d,2_a,2_c]$. We color $v_{k-1}v_{k-1}'$ with $2_b$ and all $v_iv_i'$ with $1$, $2 \le i \le k-2$.  

\textbf{Case 2.2.2:} $f(v_1v_k) = f(vv') = 1$, $f(u_1v_1) = 3$, $f(v_ku_k) = 2_a$, $f(vv_1) = 2_b$, $f(v_1v_k) = 2_c$. Then we color $v_1v_2, v_{k-1}v_k$ with $2_b, 2_c$ and $v_{k-2}v_{k-1}$ until $v_2v_3$ using the pattern $[2_d,2_a,2_c]$. Furthermore, we color all $v_iv_i'$ with $1$, $2 \le i \le k-1$.

\textbf{Case 2.2.3:} $f(v_ku_k) = f(vv_1) = 1$, $f(v_1v_k) = 2_a$, $f(u_1v_1) = 3$, $f(vv_k) = 2_b$. Then we color $v_1v_2, v_{k-1}v_k$ with $1,2_b$ and $v_{k-2}v_{k-1}$ until $v_2v_3$ together with $v_2v_2'$ using the pattern $[2_d, 2_c, 2_b]$. Furthermore, we color $v_iv_i'$ with $1$ for $3 \leq i \leq k-1$.

\textbf{Case 2.3:} Color $1$ is not used on $u_1v_1, v_1v_k, v_ku_k$. There are two subcases.

\textbf{Case 2.3.1:} $f(u_1v_1) = 2_a$, $f(v_1v_k) = 2_b$, $f(v_ku_k) = 3$, $\{f(vv_1), f(v_1v_k),f(vv')\} = \{1,2_c,2_d\}$. Then we can always exchange colors so that $v_1v_k$ is colored with $1$, which is already solved in \textbf{Case 2.2}.

%$f^{-1}(1)\cap \{vv_1, ...\}$
%Then we color $v_1v_2, v_{k-1}v_k$ with $1,1$ and $v_{k-2}v_{k-1}$ until $v_2v_3$ together with $v_2v_2'$ using the pattern $[2_c,2_b,2_d]$. Furthermore, we color $v_2v_2'$ with $2_d$ and all $v_iv_i'$ with $1$, $3 \le i \le k-1$.

\textbf{Case 2.3.2:} $f(u_1v_1) = 2_a$, $f(v_1v_k) = 3$, $f(v_ku_k) = 2_b$, $\{f(vv_1) f(v_1v_k),f(vv')\} = \{1,2_c,2_d\}$. We color $v_1v_2, v_{k-1}v_k$ with $1$ and $v_{k-2}v_{k-1}, \ldots, v_2v_3$ together with $v_2v_2'$ using the pattern $[2_c,2_b,2_d]$. Furthermore, we color $v_{k-1}v_{k-1}'$ with $2_a$ and all $v_iv_i'$ with $1$, $3 \le i \le k-2$.
\end{proof}

Let $P=F_0F_{\ell} \ldots F_2F_1$ be a longest path in the weak dual of $B_1$ starting at $F_0$. We first show $\ell \ge 2$. 

%Then we describe the ``outermost two layers'' of the block $B_1$. 

\begin{lemma}\label{122223-two-layers}
$\ell \ge 2$.   
\end{lemma}

\begin{proof}
We first show it is impossible that $\ell = 1$ and $F_2$ (now also $F_0$) is a $3$-face. Suppose not, i.e., $F_2 = F_0$ and $F_2$ has boundary cycle $u_0u_1u_2$, where $u_1,u_2 \in V(F_1) \cap V(F_2)$. In this case, $F_1$ is a $3$-face or a $4$-face by Lemma~\ref{122223-outside-1}. We delete $V(F_1)$ with all its leaf edges from $G$ to obtain a graph $G'$. By the minimality of $G$, $G'$ has a good coloring $f$. By the definition of good coloring, we may assume $f(v_0u_0)$ is either $1$ or $2_a$. In case $F_1$ is a $3$-face, let the boundary cycle of $F_1$ be $u_1u_2u_3$ and $u_3u_3'$ be the leaf edge. If $f(v_0u_0) = 1$, then we can always color $u_0u_1$ and $u_0u_2$ with colors from $\{2_a, 2_b, 2_c, 2_d\}$, say $u_0u_1$ with $2_c$ and $u_0u_2$ with $2_d$. Then we color $u_1u_3, u_2u_3, u_3u_3'$ with $2_a, 2_b, 1$. If $f(v_0u_0) = 2_a$, then we can always color $u_0u_2$ with a color in $\{2_b, 2_c, 2_d\}$, say $2_b$, and color $u_0u_1,u_1u_2, u_1u_3, u_2u_3, u_3u_3'$ with $1,2_d,2_c, 1, 2_a$. This is a contradiction. In case $F_1$ is a $4$-face,, let the boundary cycle of $F_1$ be $u_1u_2u_3u_4$ with two leaf edges $u_3u_3'$ and $u_4u_4'$. If $f(v_0u_0) = 1$, then we can always color $u_0u_1, u_0u_2$ with two colors from $\{2_a, 2_b, 2_c, 2_d\}$, say $2_c,2_d$, and color $u_1u_2, u_1u_4, u_2u_3, u_3u_4, u_3u_3', u_4u_4'$ with $1,2_a,2_b,1,2_c, 2_d$. If $f(v_0u_0) = 2_a$, then we can always color $u_0u_2$ with a color in $\{2_b, 2_c, 2_d\}$, say $2_b$, and color $u_0u_1, u_1u_2, u_1u_4, u_2u_3, u_3u_4, u_3u_3', u_4u_4'$ with $1,3,2_c,1, 2_d, 2_a, 2_b$. This is a contradiction.

Then we complete the proof that $\ell \neq 1$. Suppose not, let the boundary cycle of $F_2$ be $u_0u_1 \ldots u_k$, $k \ge 3$. We delete $u_1, \ldots, u_k$, all pendant faces adjacent to $F_2$, and their leaf edges, to obtain a graph $G'$. By the minimality of $G$, $G'$ has a good coloring $f$. Let $F$ be a pendant face adjacent to $F_2$ with $V(F) \cap V(F_2) = \{u_i,u_{i+1}\}$, $1 \le i \le k-1$. We may treat $F$ as a $4$-face $u_iu_i'u_{i+1}'u_{i+1}u_i$ with an additional pendant face $u_i'u_{i+1}'u_i''$ plus a leaf edge $u_i''w_i$ (neither a $3$-face $u_iu_i'u_{i+1}u_i$ with leaf edge $u_i'u_i''$ nor a $4$-face $u_iu_i'u_{i+1}'u_{i+1}u_i$ with leaf edges $u_i'u_i'', u_{i+1}'u_{i+1}''$) when we extend the coloring $f$ to $G$, since we can always use the coloring of the HC graph (induced by $w_i, u_{i-1}, u_i, u_i', u_i'', u_{i+1}, u_{i+1}', u_{i+1}'', u_{i+2}$) on $F$ and its leaf edges. Note that we always color $u_iu_i', u_{i+1}u_{i+1}'$ with different colors and can use the color used on $u_i'u_{i+1}'$ (of the HC graph) to color the leaf edge $u_i'u_i''$ of a $3$-face. Let $F_1'$ be the pendant face adjacent to $F_2$ such that $V(F_1') \cap V(F_2)$ has the largest indices, say $V(F_1') \cap V(F_2) = \{u_j,u_{j+1}\}$, where $1 \le j \le k-1$. We know $k \ge 3$ and $f(u_0v_0) = 1$ or $f(u_0v_0) \in \{2_a, 2_b, 2_c, 2_d\}$, say $f(u_0v_0) = 2_c$. We can first color every leaf edge hang on $F_2$ with $1$. %There are two cases depending on the color used on $u_0v_0$.HC graph

\textbf{Case A:} $f(u_0v_0) = 1$. We can always color $u_0u_1, u_0u_k$ with two colors from $\{2_a, 2_b, 2_c, 2_d\}$, say $2_a,2_b$. If $k \equiv 0 \text{ (or 2)} \text{ }(\text{mod } 3)$, then we use the pattern $[2_c,2_d,2_a]$ to color the edges from $u_1u_2$ to $u_{k-1}u_k$. We color $u_ju_j'$ with $f(u_ju_{j+1})$, recolor $u_ju_{j+1}$ to $3$, and color $u_{j+1}u_{j+1}', u_j'u_{j+1}', u_j'u_j'', u_{j+1}'u_j'', u_j''w_j$ with $1,2_a (2_d), 1, f(u_{j-1}u_j), 2_b$ when $j = k-1$ and with $1,2_b, 1, f(u_{j-1}u_j), f(u_{j+1}u_{j+2})$ when $j \le k-2$. If $k \equiv 1 \text{ }(\text{mod } 3)$, then we use the pattern $[2_c,2_d,2_a]$ to color the edges from $u_1u_2$ to $u_{k-1}u_k$ except $u_ju_{j+1}$, and color $u_ju_{j+1}, u_ju_j', u_{j+1}u_{j+1}', u_j'u_{j+1}', u_j'u_j'', u_{j+1}'u_j'', u_j''w_j$ with $3,2_a,1,2_c,1,2_d, 2_b$ when $j = k-1$ and color $u_ju_{j+1}, u_ju_j', u_{j+1}u_{j+1}', u_j'u_{j+1}',$ $ u_j'u_j'', u_{j+1}'u_j'', u_j''w_j$ with $3,1,x,f(u_{j+2}u_{j+3}),$ $f(u_{j+1}u_{j+2}),1,y$ when $j \le k-2$, where $x \notin \{f(u_{j-1}u_{j}), f(u_{j+1}u_{j+2}),f(u_{j+2}u_{j+3}), 1,3\}$ and $y \notin \{f(u_j'u_{j+1}'), f(u_j'u_j''),x, 1,3\}$. For other HC graphs on $u_i, u_{i+1}$, where $1 \le i \le k-3$ and $i \neq j$, we color $u_iu_i'$ with $f(u_iu_{i+1})$, recolor $u_iu_{i+1}$ to $1$, and color $u_{i+1}u_{i+1}', u_i'u_{i+1}', u_i'u_i'', u_{i+1}'u_i'', u_i''w_i$ with $2_b,1, f(u_{i+1}u_{i+2}), f(u_{i-1}u_i),1$. %This is a contradiction.

\textbf{Case B:} $f(u_0v_0) = 2_c$. We may assume $1 \in \{f(v_0v_0'), f(v_0v_0'')\}$ and thus can color $u_0u_1, u_0u_k$ with two colors from $\{2_a, 2_b, 2_c, 2_d\}$, say $2_a,2_b$. If $k \equiv 2 \text{ }(\text{mod } 3)$, then we use the same argument in \textbf{Case A}. If $k \equiv 0 \text{ }(\text{mod } 3)$, then we use the pattern $[2_d,2_c,2_a]$ to color the edges from $u_1u_2$ to $u_{k-1}u_k$ except $u_ju_{j+1}$, color $u_ju_{j+1}, u_ju_j', u_{j+1}u_{j+1}', u_j'u_{j+1}', u_j'u_j'',$ $ u_{j+1}'u_j'', u_j''w_j$ with $3,2_c,1,2_a,1,2_d, 2_b$ when $j = k-1$, with $3,1,2_c,2_b,2_d,1,2_a$ when $j = k-2$, and with $3,2_b,1,$ $f(u_{j+2}u_{j+3}),1,f(u_{j-1}u_j), f(u_{j+1}u_{j+2})$ when $j \le k-3$. If $k \equiv 1 \text{ }(\text{mod } 3)$, then we use the pattern $[2_d,2_c,2_a]$ to color the edges $u_1u_2$ until $u_{k-2}u_{k-1}$, and color $u_ju_{j+1}, u_ju_j', u_{j+1}u_{j+1}', u_j'u_{j+1}', u_j'u_j'', u_{j+1}'u_j'', u_j''w_j$ with $3,2_a,1,2_d,1,2_c, 2_b$ when $j = k-1$ and color $u_ju_j'$ with $f(u_ju_{j+1})$, recolor $u_{k-1}u_k, u_0u_k, u_ku_k'$ with $2_b, 1, 2_d$, color $u_{j+1}u_{j+1}', u_j'u_{j+1}', u_j'u_j'', u_{j+1}'u_j'', u_j''w_j$ with $1, x, 1, f(u_{j-1}u_j), y$ when $j \le k-2$, where $x \notin \{1,3,f(u_{j-1}u_j), f(u_ju_{j+1}), f(u_{j+1}u_{j+2})\}$ and $y \notin \{f(u_ju_j'), f(u_{j+1}'u_j''),x,$ $ 1,3\}$, and recolor $u_ju_{j+1}$ with $3$. For other HC graphs on  $u_i, u_{i+1}$, where $1 \le i \le k-3$ and $i \neq j$, we color $u_iu_i'$ with $f(u_iu_{i+1})$, recolor $u_iu_{i+1}$ to $1$, and color $u_{i+1}u_{i+1}', u_i'u_{i+1}', u_i'u_i'', u_{i+1}'u_i'', u_i''w_i$ with $2_b,1, f(u_{i+1}u_{i+2}), f(u_{i-1}u_i),1$.    
\end{proof}

Then we show $F_2$ must be a $4$-face when $\ell \ge 2$.

\begin{lemma}\label{122223-outside-2}
$F_2$ is a $4$-face. 
\end{lemma}

\begin{proof}
By Lemma~\ref{122223-two-layers}, $\ell \ge 2$. Suppose $F_2$ is a $k$-face, where $k \ge 5$. Let the boundary cycle of $F_2$ be $v_1 \ldots v_k$ with $V(F_3) \cap V(F_2) = \{v_1, v_k\}$. We delete $v_2, \ldots, v_{k-1}$ and all pendant faces hang on $F_2$ with all their leaf edges, and add vertices $z,z_1,z_2,z_k$ with edges $z_1v_1,z_kv_k,z_1z_k,z_1z_2,z_2z_k,z_2z$ (HC graph) to obtain a graph $G'$. By the minimality of $G$, $G'$ has a good coloring $f$. We aim to extend $f$ to $G'$. Let $F$ be a pendant face adjacent to $F_2$ with $V(F) \cap V(F_2) = \{v_i,v_{i+1}\}$, $2 \le i \le k-1$. We may assume $F$ is a $4$-face $v_iv_i'v_{i+1}'v_{i+1}v_i$ with leaf edges $v_i'v_i'', v_{i+1}'v_{i+1}''$ for the same reason as we discussed in the second paragraph of the proof of Lemma~\ref{122223-two-layers}. Furthermore, we may assume $v_i''=v_{i+1}''$ and add a new vertex $w_i$ with the edge $w_iv_i''$. In this way, we are coloring a larger graph than $G'$ and can use the proof in the future. Let $F_1'$ be the pendant face adjacent to $F_2$ such that $V(F_1') \cap V(F_2)$ has the largest indices, say $V(F_1') \cap V(F_2) = \{v_j,v_{j+1}\}$, where $2 \le j \le k-2$. Since there are no $2$-vertices, each $v_i$ with no pendant faces on top has a leaf neighbour $v_i'$, $2 \le i \le k-1$. Let $N(v_1) = \{u_1, v_2, v_k\}$ and $N(v_k) = \{u_k,v_1, v_{k-1}\}$. 

\textbf{Case 1:} The HC graph uses type I coloring, i.e., $u_1v_1,v_1v_k,v_ku_k,v_1z_1,v_kz_k,z_1z_k,z_1z_2,z_2z_k,z_2z$ are colored with $2_a, 1, 2_b, 2_c, 2_d, 1, 2_b, 2_a, 1$. We color $v_1v_2, v_{k-1}v_k$ with $2_c,2_d$ and all leaf edges $v_iv_i'$ with $1$. 

\textbf{Case 1.1:} $k \equiv 2 \text{ }(\text{mod } 3)$. Then we use the pattern $[2_b,2_a,2_c]$ to color the edge $v_2v_3$ until the edge $v_{k-2}v_{k-1}$. If there is an HC graph on $v_{k-2},v_{k-1}$,  then we color $v_{k-2}v_{k-1}, v_{k-2}v_{k-2}', v_{k-1}v_{k-1}', v_{k-2}'v_{k-1}', v_{k-2}'v_{k-2}'', v_{k-1}'v_{k-2}'', v_{k-2}''w_{k-2}$ with $1, 2_a, 2_c, 1, 2_d, 2_b, 1$. For other HC graphs on $v_{i},v_{i+1}$, where $2 \le i \le k-3$, we color $v_{i+1}v_{i+1}'$ with $f(v_iv_{i+1})$, recolor $v_iv_{i+1}$ with $1$, color $v_iv_i', v_i'v_{i+1}', v_i'v_i'', v_{i+1}'v_i'', v_i''w_i$ with $2_d, 1, f(v_{i+1}v_{i+2}), f(v_{i-1}v_{i}), 1$.

\textbf{Case 1.2:} $k \equiv 1 \text{ }(\text{mod } 3)$. Since $k \ge 5$, $k \ge 7$. Then we use the pattern $[2_b,2_a,2_c]$ to color the edge $v_2v_3$ until the edge $v_{k-2}v_{k-1}$ except $v_jv_{j+1}$. If $j = k-2$, then we color $v_{k-2}v_{k-1}, v_{k-2}v_{k-2}', v_{k-1}v_{k-1}', v_{k-2}'v_{k-1}', v_{k-2}'v_{k-2}'', v_{k-1}'v_{k-2}'', v_{k-2}''w_{k-2}$ with $1, 2_b, 2_a, 1, 2_d, 2_c, 1$. If $3 \le j \le k-3$, then we color  $v_iv_{i+1},v_iv_i', v_{i+1}v_{i+1}', v_i'v_{i+1}', v_i'v_i'', v_{i+1}'v_i'', v_i''w_i$ with $3, 2_d, 1, $ $ f(v_{j-2}v_{j-1}), 1, f(v_{j-1}v_j), f(v_{j+1}v_{j+2})$. If $j = 2$, then we use the pattern $[2_a,2_b,2_d]$ to color the edge $v_{k-2}v_{k-1}$ until the edge $v_3v_4$, and color $v_2v_3, v_2v_2', v_3v_3', v_2'v_3', v_2'v_2'', v_2''v_3', v_2''w_2$ with $1, 2_b, 2_a, 1, 2_d, 2_c, 1$. For other HC graphs on $v_{i},v_{i+1}$, where $2 \le i \le k-3$ and $i \neq j$, we color $v_{i+1}v_{i+1}'$ with $f(v_iv_{i+1})$, recolor $v_iv_{i+1}$ with $1$, color $v_iv_i', v_i'v_{i+1}', v_i'v_i'', v_{i+1}'v_i'', v_i''w_i$ with $2_d, 1, f(v_{i+1}v_{i+2}), f(v_{i-1}v_{i}), 1$.

\textbf{Case 1.3:} $k \equiv 0 \text{ }(\text{mod } 3)$. Since $k \ge 5$, $k \ge 6$. We use the pattern $[2_d, 2_a, 2_c]$ to color the edge $v_2v_3$ until $v_{k-2}v_{k-1}$. If there is a HC on $v_{i},v_{i+1}$, where $2 \le i \le k-2$, we color $v_{i+1}v_{i+1}'$ with $f(v_iv_{i+1})$, recolor $v_iv_{i+1}$ with $1$, color $v_iv_i', v_i'v_{i+1}', v_i'v_i'', v_{i+1}'v_i'', v_i''w_i$ with $2_b, 1, f(v_{i+1}v_{i+2}), f(v_{i-1}v_{i}), 1$.

%If $j = k-2$, then we use the pattern $[2_d, 2_a, 2_c]$ to color the edge $v_2v_3$ until $v_{k-3}v_{k-2}$, color $v_{k-2}v_{k-1}, v_{k-2}v_{k-2}', v_{k-1}v_{k-1}', v_{k-2}'v_{k-1}', v_{k-2}'v_{k-2}'', v_{k-1}'v_{k-2}'', v_{k-2}''w_{k-2}$ with $1, 2_b, 2_c, 1, 2_d, 2_a, 1$. For other HC graphs on $v_{i},v_{i+1}$, where $2 \le i \le k-3$ and $i \neq j$, we color $v_{i+1}v_{i+1}'$ with $f(v_iv_{i+1})$, recolor $v_iv_{i+1}$ with $1$, color $v_iv_i', v_i'v_{i+1}', v_i'v_i'', v_{i+1}'v_i'', v_i''w_i$ with $2_b, 1, f(v_{i+1}v_{i+2}), f(v_{i-1}v_{i}), 1$. Otherwise ($j \le k-3$), we use the pattern $[2_b, 2_a, 2_c]$ to color the edge $v_2v_3$ until $v_{k-2}v_{k-1}$. If there is a HC graph on $v_{k-3},v_{k-2}$, then we color $v_{k-2}v_{k-2}'$ with $f(v_{k-3}v_{k-2})$, recolor $v_{k-3}v_{k-2}$ with $3$, and color $v_{k-3}v_{k-3}', v_{k-3}'v_{k-2}', v_{k-3}'v_{k-3}'', v_{k-2}'v_{k-3}'', v_{k-3}''w_{k-3}$ with $1, 2_d, 2_c, 1, 2_b$. For other HC graphs on $v_{i},v_{i+1}$, where $2 \le i \le k-4$, we color $v_{i+1}v_{i+1}'$ with $f(v_iv_{i+1})$, recolor $v_iv_{i+1}$ with $1$, color $v_iv_i', v_i'v_{i+1}', v_i'v_i'', v_{i+1}'v_i'', v_i''w_i$ with $2_d, 1, f(v_{i+1}v_{i+2}), f(v_{i-1}v_{i}), 1$.

\textbf{Case 2:} The HC graph uses type II coloring, i.e., $u_1v_1,v_1v_k,v_ku_k,v_1z_1,v_kz_k,z_1z_k,z_1z_2,z_2z_k,z_2z$ are colored with $2_a, 3, 2_b, 2_c, 1, 2_d, 2_b, 2_a, 1$. We color $v_1v_2, v_{k-1}v_k$ with $2_c, 1$, all leaf edges $v_iv_i'$ with $1$ except $v_{k-1}v_{k-1}'$ (if it is a leaf), and use the pattern $[2_d, 2_a, 2_c]$ to color the edge $v_2v_3$ until $v_{k-2}v_{k-1}$ together with $v_{k-1}v_{k-1}'$ (if it is a leaf). Since $k \ge 5$, if there is an HC graph on $v_{k-2}, v_{k-1}$, then we color $v_{k-1}v_{k-1}'$ with $f(v_{k-2}v_{k-1})$, recolor $v_1v_k, v_{k-1}v_k, v_{k-2}v_{k-1}$ with $1,3,1$, and color $v_{k-2}v_{k-2}', v_{k-2}'v_{k-1}', v_{k-2}'v_{k-2}'', v_{k-1}'v_{k-2}'', v_{k-2}''w_{k-2}$ with $2_b, 1, f(v_{k-4}v_{k-3}), $ $f(v_{k-3}v_{k-2}), 1$. For other HC graphs on $v_{i}, v_{i+1}$, $2 \le i \le k-3$, we color $v_{i+1}v_{i+1}'$ with $f(v_iv_{i+1})$, recolor $v_iv_{i+1}$ with $1$, color $v_iv_i', v_i'v_{i+1}', v_i'v_i'', v_{i+1}'v_i'', v_i''w_i$ with $2_b, 1, f(v_{i+1}v_{i+2}), f(v_{i-1}v_{i}), 1$.

\textbf{Case 3:} The HC graph uses type III coloring, i.e., $u_1v_1,v_1v_k,v_ku_k,v_1z_1,v_kz_k,z_1z_k,z_1z_2,z_2z_k,z_2z$ are colored with $2_c, 1, 3, 2_a, 2_b, 1, 2_d, 2_c, 1$. We color $v_1v_2, v_{k-1}v_k$ with $2_a,2_b$, all leaf edges $v_iv_i'$ with $1$, and use the pattern $[2_d, 2_c, 2_a]$ to color the edge $v_2v_3$ until $v_{k-2}v_{k-1}$. Since $k \ge 5$, if there is an HC graph on $v_{k-2}, v_{k-1}$, then we color $v_{k-2}v_{k-2}'$ with $f(v_{k-2}v_{k-1})$, recolor $v_{k-2}v_{k-1}$ with $1$, and color $v_{k-1}v_{k-1}', v_{k-2}'v_{k-1}', v_{k-2}'v_{k-2}'', v_{k-1}'v_{k-2}'', v_{k-2}''w_{k-2}$ with $f(v_{k-4}v_{k-3}), 1, 2_b,$ $ f(v_{k-3}v_{k-2}), 1$. For other HC graphs on $v_{i}, v_{i+1}$, $2 \le i \le k-3$, we color $v_{i+1}v_{i+1}'$ with $f(v_iv_{i+1})$, recolor $v_iv_{i+1}$ with $1$, color $v_iv_i', v_i'v_{i+1}', v_i'v_i'', v_{i+1}'v_i'', v_i''w_i$ with $2_b, 1, f(v_{i+1}v_{i+2}), f(v_{i-1}v_{i}), 1$. 
\end{proof}

%We first show if $F_1$ is a triangle then there can be no edge hang on $F_1$. Suppose not, i.e., $F_1 = u_1u_2u_3$ and $u_1u_1'$ is the leaf edge hang on $F_1$. Since $B_1$ has at least two faces, we may assume $N(u_2) = \{u_1,u_3,u_4\}$ and $N(u_3) = \{u_1, u_2, u_5\}$. We delete $u_1u_1'$ to obtain a graph $G'$. By the minimality of $G$, $G'$ has a good coloring $f$. We may assume $f(u_1u_2) = 1, f(u_1u_3) = 2_a, f(u_2u_3) = 2_b, f(u_2u_4) = 2_c, f(u_3u_5) = 2_d$ since otherwise we can extend $f$ to $G$ by coloring $u_1u_1'$ with a color in $\{1, 2_a, 2_b, 2_c, 2_d\}$. Then we switch the colors used on $u_1u_2$ and $u_2u_3$, and color $u_1u_1'$ with $1$. This is a contradiction.

\begin{rem}\label{122223-outside-2-rem}
Additionally, in Lemma~\ref{122223-outside-2}, if $F_2$ is a $4$-face, then $F_1$ must be a $3$-face. Suppose not, i.e., $F_1$ is a $4$-face $v_1v_2v_3v_4v_1$ with leaf edges $v_2v_2',v_3v_3'$, where $v(F_1) \cap V(F_2) = \{v_1, v_4\}$. Let the boundary cycle of $F_2$ be $v_1'v_1v_4v_4'v_1'$, $N(v_1') = \{v_1, v_1'', v_4\}$, and $N(v_4') = \{v_1', v_4, v_4''\}$. We delete $v_2,v_2',v_3,v_3'$ and add a vertex $v,v'$ with edges $vv_1, vv_4, vv'$ to obtain a graph $G'$. By the minimality of $G$, $G'$ has a good coloring $f$. If the HC graph (induced by $v,v',v_1,v_1', v_1'', v_4, v_4', v_4''$) uses type I coloring (type III coloring), say $v_1'v_1'', v_1'v_4',v_4'v_4'', v_1v_1', v_4v_4', v_1v_4, vv_1,vv_4, vv_1$ uses $2_a (3), 1, 2_b, 2_c, 2_d, 1, 2_b, 2_a, 1$, then we color $v_1v_2, v_3v_4, v_2v_3, v_2v_2', v_3v_3'$ with $2_b, 2_a, 1, 2_d, 2_c$. If the HC graph uses type II coloring, say $v_1'v_1'', v_1'v_4',v_4'v_4'',$ $v_1v_1', v_4v_4', v_1v_4, vv_1,vv_4, vv_1$ uses $2_a, 3, 2_b, 1, 2_c, 2_d, 2_b, 1, 2_a$, then we recolor $v_1'v_4', v_1v_1', v_1v_4$ with $1,3,1$, and color $v_1v_2,$ $ v_3v_4, v_2v_3, v_2v_2', v_3v_3'$ with $2_b, 2_a, 1, 2_c, 2_d$.
\end{rem}
    
Now, we complete the proof. By Lemma~\ref{122223-two-layers}, $\ell \ge 2$. We split the proof into two cases.

\textbf{Case i:} $\ell = 2$. By Lemma~\ref{122223-outside-2} together with Remark~\ref{122223-outside-2-rem}, $B_1$ can be described as $F_0$ with a few leaf edges, $3$-faces with a leaf edge, $4$-faces with two leaf edges, and HC graphs attached on it. We can use the same argument used in the proof of Lemma~\ref{122223-two-layers}, since we treated every $3$-face with a leaf edge and every $4$-face with two leaf edges as an HC graph in the coloring procedure.

% ADD A PICTURE
\textbf{Case ii:} $\ell \ge 3$. Let $V(F_3) \cap V(F_4) = \{v_1, v_k\}$. We delete $V(F_3) - \{v_1, v_k\}$, all leaf edges, all $3$-faces with a leaf edge, all $4$-faces with two leaf edges, all HC graphs attached on $F_3$, and add vertices $z,z_1,z_2,z_k$ with edges $z_1v_1,z_kv_k,z_1z_k,z_1z_2,z_2z_k,z_2z$ (HC graph) to obtain a graph $G'$. By the minimality of $G$, $G'$ has a good coloring $f$. We can use the argument used in the proof of Lemma~\ref{122223-outside-2}, since we treated every $3$-face with a leaf edge and every $4$-face with two leaf edges as an HC graph in the color-extending procedure. \hfill \qed

\section{$(1^2,2^k)$-packing edge-colorings of subcubic outerplanar graphs}

In this section, we focus on $(1^2,2^k)$-colorings of subcubic outerplanar graphs. We first show in Example~\ref{ex1} that there are subcubic outerplanar graphs that are not $(1^2,2^2)$-colorable. Note that our example is also $2$-connected.

\begin{example}\label{ex1}
Let the graphs $G_0$ and $G_1$ be the graphs shown in Fig.~\ref{example_2}. The graph $G_1$ is subcubic, outerplanar, and even $2$-connected. We show $G_1$ has no $(1^2,2^2)$-packing edge-coloring.
\end{example}

\begin{figure}[ht]
 \vspace{-8mm}
\begin{center}
  \includegraphics[scale=0.6]{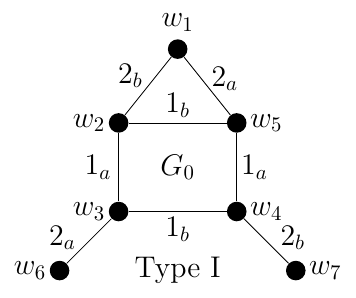} \hspace{8mm}
  \includegraphics[scale=0.6]{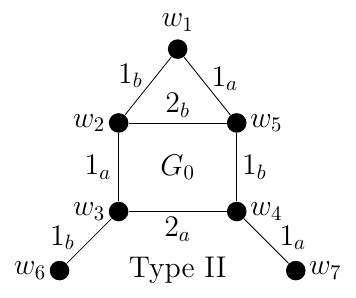} \hspace{8mm}
  \includegraphics[scale=0.53]{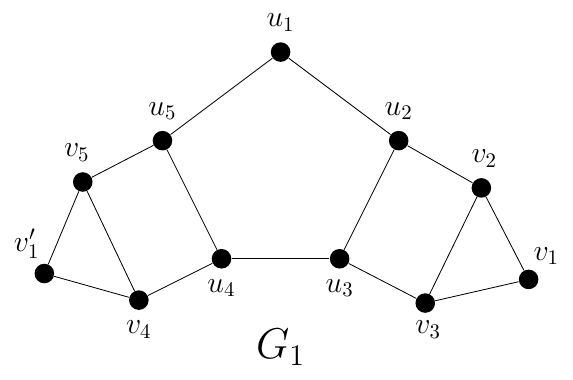} %\hspace{10mm}
 \vspace{-3mm}
\caption{A subcubic outerplanar graph that is not $(1^2,2^2)$-packing edge-colorable.}\label{example_2}
\end{center}
\vspace{-8mm}
\end{figure}

\begin{proof}
We first show graph $G_0$ has exactly two different $(1^2,2^2)$-packing edge-coloring up to symmetry. 

\begin{claim}\label{house1122}
Graph $G_0$ has exactly two different $(1^2,2^2)$-packing edge-coloring up to symmetry.
\end{claim}

\begin{proof}
Let $f$ be a $(1^2,2^2)$-packing edge-coloring of $G_0$. Since the triangle $w_1w_2w_5$ needs three distinct colors, $\{2_a, 2_b\} \nsubseteq \{f(w_2w_3), f(w_3w_4), f(w_4w_5)\}$. Therefore, the colors of $w_2w_3, w_3w_4, w_4w_5$ can be $1_a, 1_b, 1_a$ or $1_a, 2_a, 1_a$ or $1_a, 1_b, 2_a$ or $1_a, 2_a, 1_b$ up to symmetry. If it is $1_a, 1_b, 2_a$, then $f(w_3w_6)$ must be $2_b$ and we cannot color $w_1w_2w_5$, which is a contradiction. If it is $1_a, 2_a, 1_a$, then there are only two available colors ($\{1_b, 2_b\}$) for $w_1w_2w_5$, which is again a contradiction. Therefore, the colors used on $w_2w_3, w_3w_4, w_4w_5$ must be $1_a, 1_b, 1_a$ or $1_a, 2_a, 1_b$. Furthermore, if the colors are $1_a, 1_b, 1_a$, then $f(w_3w_6) = 2_a$ and $f(w_4w_7) = 2_b$, and $f(w_2w_5) = 1_b, f(w_1w_5) = 2_a$, and $f(w_1w_2) = 2_b$. We call this the \textbf{Type I} coloring of $G_0$. If the colors are $1_a, 2_a, 1_b$, then $w_2w_5$ must use $2_b$ and thus $w_1w_2, w_1w_5$ must use $1_b, 1_a$. Hence, $w_3w_6, w_4w_7$ must use $1_b,1_a$. We call this the \textbf{Type II} coloring of $G_0$.    
\end{proof}

We are now ready to show $G_1$ has no $(1^2,2^2)$-packing edge-coloring. Suppose not, i.e., $G_1$ has a $(1^2,2^2)$-packing edge-coloring $f$. Since $G_1$ contains two copies of $G_0$ and they share the edge $u_3u_4$, they must be the same type. If both copies are of Type I, then there are three edges in the $5$-cycle that are colored with a color in $\{2_a,2_b\}$. This is a contradiction. If both copies are of Type II, say $f(u_2u_3) = 2_a, f(u_1u_2) = 1_a, f(u_3u_4) = 1_b$, then it implies $f(u_4u_5) = 2_b$ and $f(u_1u_5) = 1_a$. This is again a contradiction.
\end{proof}
%We discuss two cases based on the color used on $w_3w_4$. If $w_3w_4$ is colored $1_a$, then $u_2u_3$ and $u_4u_5$ cannot both be colored with a color in $\{2_a, 2_b\}$.

Inspired by Conjecture~\ref{main-conj}, we explore the question ``What is the largest $k_2$ such that every subcubic outerplanar graph is $(1^2,2^2,k_2)$-packing edge-colorable?''. We prove in Section~\ref{1122k} that $k_2 \ge 3$ and $k_2 \le 4$. 

\begin{theorem}
Every subcubic outerplanar graph is $(1^2,2^2,3)$-packing edge-colorable. 
\end{theorem}

As a corollary of our result, every subcubic outerplanar graph is $(1^2,2^3)$-packing edge-colorable, and thus we confirmed the conjecture (two $1$-colors) of Hocquard et al~\cite{HLL} for subcubic outerplanar graphs.

\begin{cor}\label{cor-2}
Every subcubic outerplanar graph is $(1^2,2^3)$-packing edge-colorable.
\end{cor}

Corollary~\ref{cor-2} is also sharp since there are subcubic outerplanar graphs that are not $(1^2,2^2)$-colorable (Example~\ref{ex1}).

\subsection{$(1^2,2^2,k_2)$-packing edge-coloring of subcubic outerplanar graphs}\label{1122k}

We provide a subcubic outerplanar graph that is not $(1^2,2^2,5)$-colorable, showing $k_2 \le 4$.

\begin{example}
 The graph $G_3$ is shown in the right picture of Fig.~\ref{example_6}. Each of the gadget $G_{2a}, \ldots, G_{2e}$ is isomorphic to $G_2$, which is shown in the left picture of Fig.~\ref{example_6}. Each vertex $u_{1a}, \ldots, u_{1e}$ of $G_{2a}, \ldots, G_{2e}$ corresponds to $u_1$ of $G_2$. In $G_2$, each gadget $G_{1a}, \ldots, G_{1g}$ is isomorphic to the graph $G_1$ in Example~\ref{ex1}.  We show $G_3$ is a subcubic outerplanar graph with no $(1^2,2^2,5)$-coloring.
\end{example}

\begin{figure}[ht]
\begin{center}
 \vspace{-3mm}
 \includegraphics[scale=0.35]{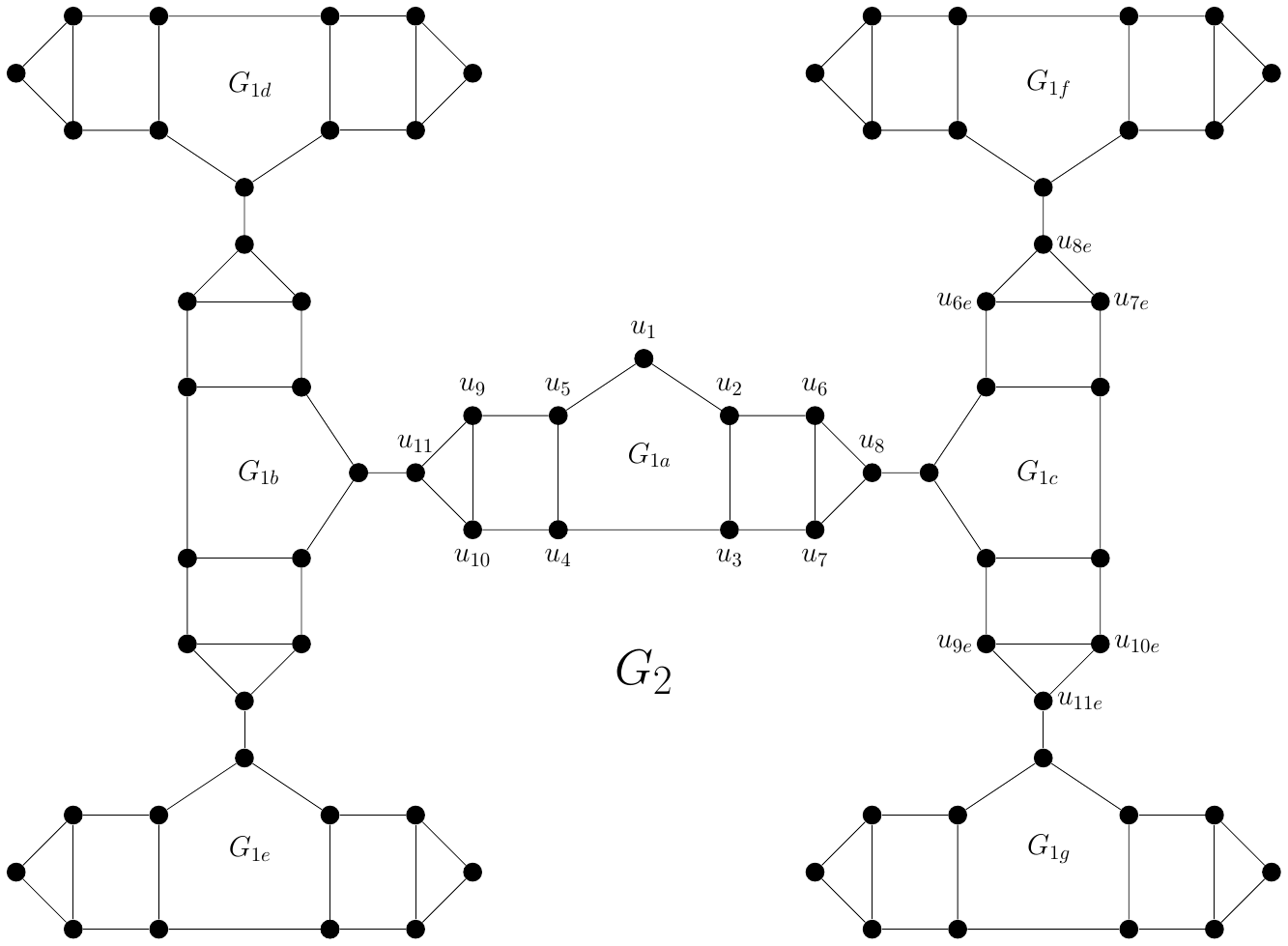}% \hspace{10mm}
  \includegraphics[scale=0.5]{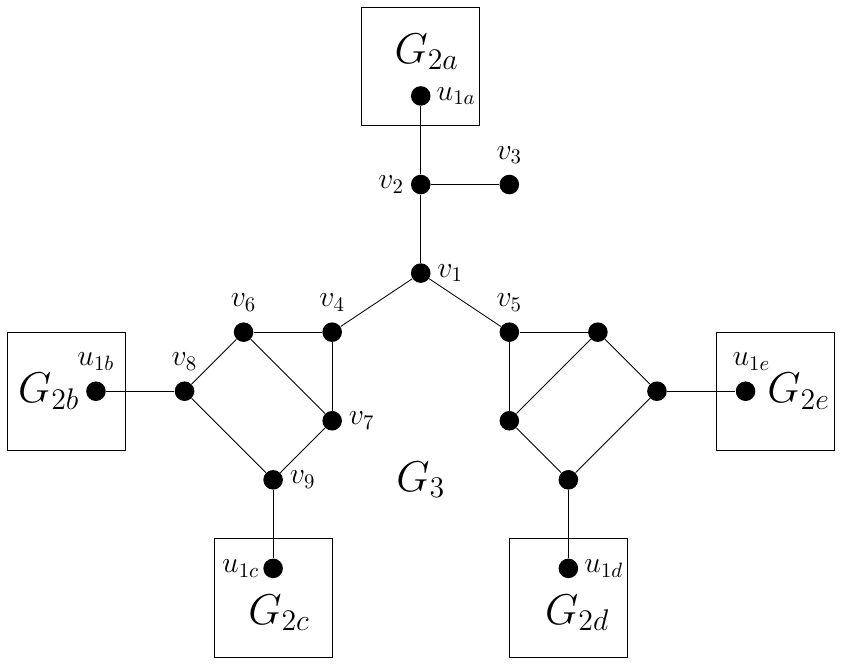} %\hspace{10mm}

\caption{A subcubic outerplanar graph that is not $(1^2,2^2,5)$-packing edge-colorable.}\label{example_6}
\end{center}
\vspace{-8mm}
\end{figure}

\begin{proof}
We claim only edges $u_1u_2, u_2u_3, u_3u_4, u_4u_5, u_1u_5$ in $G_2$ can use color $5$ in each $(1^2,2^2,5)$-coloring. 

\begin{claim}\label{only5}
Only edges $u_1u_2, u_2u_3, u_3u_4, u_4u_5, u_1u_5$ in $G_2$ can use color $5$ in each $(1^2,2^2,5)$-coloring.    
\end{claim}

\begin{proof}
By Example~\ref{ex1}, each of $G_{1a}, \ldots, G_{1g}$ must use color $5$ exactly once. We first notice that $u_6u_8,u_7u_8$ cannot be colored with $5$. Otherwise, no vertices of $G_{1e}$ can be colored with $5$. This is a contradiction. Similarly, $u_9u_{11}, u_{10}u_{11}$ cannot be colored with $5$. Suppose one of $u_2u_6, u_6u_7, u_3u_7$ is colored with $5$. Then the only two edges of $G_{1e}$ that can be colored with $5$ are $u_{7e}u_{8e}$ and $u_{10e}u_{11e}$, say $u_{7e}u_{8e}$ is colored with $5$. However, no vertices of $G_{1f}$ can be colored with $5$, which is a contradiction. Similarly, $u_5u_9, u_9u_{10}, u_4u_{10}$ cannot be colored with $5$.    
\end{proof}

Suppose $G_3$ has a $(1^2,2^2,5)$-coloring. By Claim~\ref{only5}, none of $u_{1a}v_2, v_2v_3, v_1v_2, v_1v_5, v_1v_4$ can be colored with $5$. Furthermore, $v_1v_4$ and $v_1v_5$ cannot both be colored with a color in $\{2_a, 2_b\}$, since otherwise, say $v_1v_4,v_1v_5$ is colored with $2_a,2_b$, $v_1v_2$ must be colored with a color in $\{1_a, 1_b\}$ and we only have one available color for $v_2u_{1a}$ and $v_2v_3$. This is a contradiction. We may assume $v_1v_4$ is colored with $1_a$. Applying Claim~\ref{only5} to $G_{2b}$ and $G_{2c}$, we know none of the edges $v_4v_6, v_4v_7, v_6v_7, v_6v_8, v_7v_9, v_8v_9, v_8u_{1b}, v_9u_{1c}$ can be colored with $5$. However, by Claim~\ref{house1122}, $v_4v_6,v_4v_7$ must be colored with $2_a,2_b$. This is a contradiction since we only have one available color for edges $v_1v_2$ and $v_1v_5$.
\end{proof}

Then we show every subcubic outerplanar graph is $(1^2,2^2,3)$-packing edge-colorable and thus $3 \le k_2 \le 4$. In this subsection, a {\em good coloring} is a $(1^2,2^2,3)$-packing edge-coloring such that every HC graph can only be colored with type I, II, III, IV, V colorings, unless there are at least two non-trivial block and the vertex $w_0$ of the HC graph is a degree three vertex in $G$ (see Fig.~\ref{house-with-chimney-2}).

\begin{figure}[ht]
\begin{center}
 \vspace{-3mm}
 \includegraphics[scale=0.6]{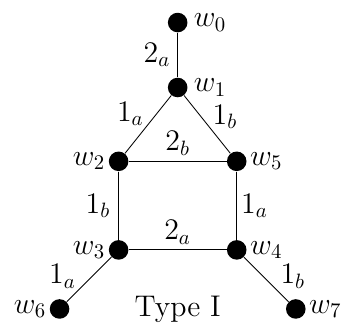} \hspace{1mm}
  \includegraphics[scale=0.6]{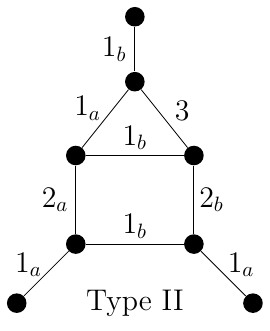} \hspace{3mm}
   \includegraphics[scale=0.6]{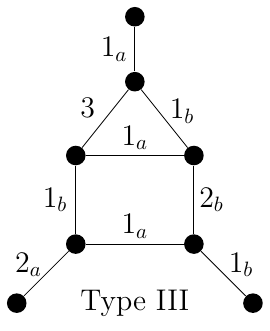} \hspace{3mm}
    \includegraphics[scale=0.6]{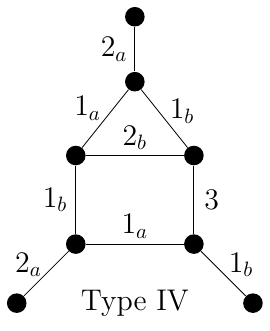} \hspace{3mm}
     \includegraphics[scale=0.6]{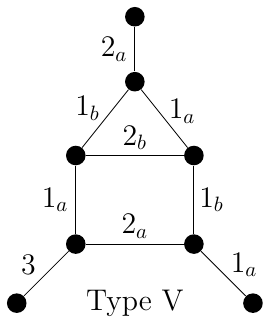}
     \vspace{-3mm}
\caption{Type I, II, III, IV, V colorings of the ``house-with-chimney'' graph.}\label{house-with-chimney-2}
\end{center}
\vspace{-8mm}
\end{figure}

%ADD WHAT IS TYPE I, II, III, IV, V

\begin{theorem}\label{11223-main-theorem}
Every subcubic outerplanar graph with no $2$-vertices has a good coloring. 
\end{theorem}

\textbf{Proof of Theorem~\ref{11223-main-theorem}:} The proof strategy is similar to the proof of Theorem~\ref{122223-main-theorem}. Suppose not, i.e., there is a subcubic outerplanar graph with no $2$-vertices and it has no good coloring. We take the graph $G$ with minimum $|V(G)| + |E(G)|$. We first show if there is a tree hanging on a cycle, then it must be an edge. To see this, let $T$ be a tree hanging on a cycle $C$ in $G$, where $u \in V(C) \cap V(T)$, and we assume $T$ has at least three edges. Let $P_1$ be a longest path in $T$ ended at $u$, say $P = uu_s \ldots u_1$, $s \ge 2$. Let the only neighbour of $u_1$ be $u_2$. Since there are no $2$-vertices, $u_2$ has another leaf neighbour $u_1'$. Since $s \ge 2$, $u_2 \neq u$. We delete $u_1,u_1'$ from $G$ to obtain a graph $G'$, which is a subcubic outerplanar graph with no $2$-vertices and satisfies $|V(G')|+|E(G')|<|V(G)|+|E(G)|$. By the minimality of $G$, $G'$ has a good coloring $f$. Let $N(u) = \{u_s, u', u''\}$. If $s = 2$, then $f(uu_2) \in \{1_a, 1_b\}$, say $f(uu_2) = 1_a$, since otherwise we can color $u_2u_1,u_2u_1'$ with $1_a, 1_b$. Furthermore, we may assume $\{f(uu'), f(uu'')\} = \{2_a, 2_b\}$, since otherwise we can color $u_2u_1'$ with $1_b$ and $u_2u_1$ with a color from $\{2_a, 2_b\}$. We also know $u', u''$ must be $3$-vertices and adjacent to two edges colored with $1_a, 1_b$, since otherwise we can recolor $uu'$ or $uu''$ with a color in $\{1_a, 1_b\}$ and complete the extension. Therefore, we can color $u_2u_1', u_2u_1$ with $1_b, 3$. If $s \ge 3$, then similarly we know $3$ is not used within distance three from $u_1u_2$ and thus we can color $u_2u_1', u_2u_1$ with $1_b, 3$. This is a contradiction. By the same proof, we may assume that $G$ has at least one cycle. In other words, $G$ has at least one non-trivial block.

We pick a non-trivial leaf block $B_1$. Let $u_0$ be the vertex of $B_1$ connecting the other blocks in $G$ and the face containing $u_0$ be $F_0$ (if there is only one non-trivial block $B_1$, then we pick any leaf edge $v_0u_0$ in a cycle of $B_1$, where $v_0$ is the $1$-vertex). We may assume $B_1$ has at least two faces. Otherwise, let $u_0u_1 \ldots u_k$ be the boundary cycle of $F_0$, where $k \ge 2$, and $N(u_0) = \{u_1, u_k, v_0\}$ and $N(v_0) = \{u_0, v_0', v_0''\}$. Since $G$ has no $2$-vertices, there is a leaf edge $u_iu_i'$ hanging on each of $u_i$, where $1 \le i \le k$. We delete all $u_i,u_i'$, $1 \le i \le k$, to obtain $G'$. By the minimality of $G$, $G'$ has a good coloring $f$. \textbf{Case (i):} $f(u_0v_0) \in \{2_a, 2_b\}$, say $f(u_0v_0) = 2_a$. If $k$ is odd, then we color $u_0u_1, \ldots, u_ku_0$ with $1_a, 1_b, 1_a, \ldots, 1_b$, and the leaf edges $u_1u_1', \ldots, u_ku_k'$ with $2_b, 2_a, \ldots, 2_b$. If $k = 2$, then we color $u_0u_1, u_1u_2, u_2u_0, u_1u_1', u_2u_2'$ with $1_a, 2_b, 1_b, 1_b, 1_a$. If $k$ is even and $k \ge 4$, then color $u_0u_1,u_1u_1',u_1u_2$ with $1_a, 1_b, 2_b$, $u_0u_k, \ldots, u_3u_2$ together with $u_2u_2'$ with $1_a, 1_b, \ldots, 1_a$, and $u_3u_3', \ldots, u_ku_k'$ with $2_a, 2_b, \ldots, 2_b$.
\textbf{Case (ii):} $f(u_0v_0) \in \{1_a, 1_b\}$, say $f(u_0v_0) = 1_a$. Then we can color $u_0u_1$ with $1_b$ and $u_0u_k$ with one of $\{2_a, 2_b, 3\}$. To see this, if $\{f(v_0v_0'), f(v_0v_0'')\} \neq \{2_a,2_b\}$, then we are done. If $\{f(v_0v_0'), f(v_0v_0'')\} = \{2_a,2_b\}$, then each of $v_0',v_0''$ sees $1_a,1_b$ and thus we can color $u_0u_k$ with $3$. In case $k$ is even, we color the edges $u_1u_2$ until $u_{k-1}u_k$ together with $u_ku_k'$ with $1_a, 1_b, \ldots, 1_a, 1_b$, and $u_1u_1', \ldots, u_{k-1}u_{k-1}'$ with $x,y, \ldots, x$, where $x \in \{2_a, 2_b\} - f(u_0u_k)$ and $y \in \{2_a, 2_b\} - x$. In case $k = 3$, we color $u_2u_2', u_2u_3, u_3u_3', u_3u_4, u_4u_4'$ with $1_a, x, 1_a, 1_b, 1_a$, where $x \in \{2_a, 2_b\} - f(u_0u_k)$. In case $k$ is odd and $k \ge 5$, we color $u_1u_1', u_1u_2, u_2u_2'$ with $1_a,x, 1_a$, where $x \in \{2_a, 2_b\} - f(u_0u_k)$, $u_2u_3$ until $u_{k-1}u_k$ together with $u_ku_k'$ with $1_b, 1_a, \ldots, 1_b, 1_a$, and $u_3u_3', \ldots, u_{k-1}u_{k-1}'$ with $y,x,\ldots,x$.
\textbf{Case (iii):} $f(u_0v_0) = 3$. We color $u_0u_1, u_1u_1', u_ku_0$ with $1_a, 1_b, 1_b$. In case $k =2$, we color $u_1u_2, u_2u_2'$ with $2_a, 1_a$. In case $k \ge 3$, we color $u_1u_2$ with $2_a$, $u_{k-1}u_k$ until $u_2u_3$ together with $u_2u_2'$ with $1_a, 1_b, \ldots$, and $u_3u_3'$ until $u_ku_k'$ with $2_b, 2_a, \ldots$. This is a contradiction.

We first show every pendant face (except $F_0$) in $B_1$ is a triangle.
%Explain what do we mean by saying "see".

\begin{lemma}\label{11223-outside-1}
Let $F_1$ be a pendant face in $B_1$ and $F_1 \neq F_0$. Then $F_1$ is a $3$-face.    
\end{lemma}

\begin{proof}
Suppose $F_1$ is a $k$-face, where $k \ge 4$. Let the only face adjacent to $F_1$ be $F_2$ and the boundary cycle of $F_1$ be $v_1 \ldots v_k$ with $V(F_1) \cap V(F_2) = \{v_1, v_k\}$. Since there are no $2$-vertices, each $v_i$ has a leaf neighbour $v_i'$, $2 \le i \le k-1$. Let $N(v_1) = \{u_1, v_2, v_k\}$ and $N(v_k) = \{u_k,v_1, v_{k-1}\}$. We delete all $v_i,v_i'$, $2 \le i \le k-1$, and add a vertex $v,v'$ with edges $vv_1, vv_k, vv'$ to obtain a graph $G'$. By the minimality of $G$, $G'$ has a good coloring $f$. 

\textbf{Case 1:} Color $3$ is not used on $u_1v_1, v_1v_k, v_ku_k$.

\textbf{Case 1.1:} Colors $2_a,2_b$ are not used on $u_1v_1, v_1v_k, v_ku_k$. Say $f(u_1v_1), f(v_1v_k), f(v_ku_k)$ = $\{1_a, 1_b, 1_a\}$. By symmetry, we may assume either $f(vv_1), f(vv_k)$ = $2_a, 2_b$ or $2_a, 3$. In case $f(vv_1), f(vv_k)$ = $2_a, 2_b$, then we color $v_2v_2', v_1v_2, v_{k-1}v_k$ with $1_b, 2_a, 2_b$, $v_2v_3$ until $v_{k-2}v_{k-1}$ together with $v_{k-1}v_{k-1}'$ with $1_a, 1_b, \ldots$. Note that if $k=4$ then we already completed the extension. If $k$ is even and $k \ge 6$, then we color $v_3v_3'$ until $v_{k-2}v_{k-2}'$ with $2_b, 2_a, \ldots, 2_b, 2_a$. If $k$ is odd, then we know $k \ge 5$ and color $v_3v_3'$ with $3$, and (if $k \ge 7$) $v_{k-2}v_{k-2}'$ until $v_4v_4'$ with $2_a, 2_b, \ldots, 2_a, 2_b$.

\textbf{Case 1.2:} Colors $2_a,2_b$ are used on $u_1v_1, v_1v_k, v_ku_k$. Then $2_a,2_b$ can be used either once or twice on $u_1v_1, v_1v_k, v_ku_k$.

\textbf{Case 1.2.1:} Colors $2_a,2_b$ are used once on $u_1v_1, v_1v_k, v_ku_k$. By symmetry, $u_1v_1, v_1v_k, v_ku_k$ are colored with either $1_a,2_a,1_a$, or $1_a, 2_a, 1_b$, or $1_a, 1_b, 2_a$. In the former case, we can always recolor $v_1v_k$ with $1_b$ (for all possibilities of $f(vv_1), f(vv_k), f(vv')$), which was solved in \textbf{Case 1.1}. In the middle case, we can assume $f(vv_1), f(vv_k), f(vv')$ are colored with $1_b, 1_a, 2_b$ (if not, then recolor it to this coloring). We color $v_1v_2, v_{k-1}v_k$ with $1_b, 1_a$. If $k$ is odd, then $k \ge 5$ and color $v_2v_3$ until $v_{k-2}v_{k-1}$ with $1_a, 1_b, \ldots, 1_a, 1_b$, and $v_2v_2'$ until $v_{k-1}v_{k-1}'$ with $2_b, 2_a, \ldots, 2_a, 2_b$. If $k$ is even, then color $v_2v_2', v_2v_3, v_3v_3'$ with $1_a, 2_b, 1_b$, and (if $k \ge 6$) $v_3v_4$ until $v_{k-2}v_{k-1}$ with $1_a, 1_b, \ldots, 1_a, 1_b$ and $v_4v_4'$ until $v_{k-1}v_{k-1}'$ with $2_a, 2_b, \ldots, 2_a, 2_b$. In the latter case, we know by symmetry that $f(vv_1), f(vv_k), f(vv')$ are colored with $2_b, 1_a, 1_b$ or $3, 1_a, 1_b$. We claim $3, 1_a, 1_b$ is impossible, since otherwise we color $v_{k-1}v_k$ until $v_2v_3$ together with $v_2v_2'$ with $1_a, 1_b, \ldots$ and color $v_{k-1}v_{k-1}'$ until $v_3v_3'$ with $2_b, 2_a, \ldots$. Thus, $f(vv_1), f(vv_k), f(vv') = 2_b, 1_a, 1_b$. We color $v_1v_2, v_{k-1}v_k$ with $2_b, 1_a$. If $k = 4$, then we know $u_k$ must see $3$, since otherwise we color $v_2v_2', v_2v_3, v_3v_3'$ with $1_a, 1_b, 3$. We also know $u_k$ must see $1_a$ since otherwise we switch the colors of $v_{k-1}v_k$ and $v_ku_k$ and color $v_2v_2', v_2v_3, v_3v_3'$ with $1_a, 1_b, 1_a$. Therefore, we switch the colors of $v_1v_2$ and $v_1v_k$, recolor $v_ku_k$ with $1_b$, and color $v_2v_2', v_2v_3, v_3v_3'$ with $1_a, 2_a, 1_b$. If $k \ge 5$, then color $v_3v_3'$ with $3$, $v_{k-1}v_k$ until $v_2v_3$ together with $v_2v_2'$ with $1_a, 1_b, \ldots,$ and $v_{k-1}v_{k-1}'$ until $v_4v_4'$ with $2_b, 2_a, \ldots$.

\textbf{Case 1.2.2:} Colors $2_a,2_b$ are used twice on $u_1v_1, v_1v_k, v_ku_k$. By symmetry, $u_1v_1, v_1v_k, v_ku_k$ are colored with either $1_a, 2_a, 2_b$ or $2_a, 1_a, 2_b$. In the former case, we can always recolor $v_1v_k$ to $1_b$ (for all possibilities of $f(vv_1), f(vv_k), f(vv')$), which was solved in \textbf{Case 1.2.1}. In the latter case, we can assume $f(vv_1), f(vv_k), f(vv') = 1_b, 3, 1_a$ by symmetry. We color $v_{k-1}v_k$ with $3$, $v_1v_2$ until $v_{k-2}v_{k-1}$ together with $v_{k-1}v_{k-1}'$ with $1_b, 1_a, \ldots$, and $v_2v_2'$ until $v_{k-2}v_{k-2}'$ with $2_b, 2_a, \ldots$.

\textbf{Case 2:} Color $3$ is used on $u_1v_1, v_1v_k, v_ku_k$. Then colors $2_a,2_b$ can be used twice, once, or zero times on $u_1v_1, v_1v_k,$ and $v_ku_k$. Note that if colors $2_a,2_b$ are used twice, then by symmetry $u_1v_1, v_1v_k, v_ku_k$ are colored with $2_a, 3, 2_b$ or $2_a,2_b, 3$, both of which are impossible since we only have two available colors $1_a,1_b$ for $vv_1, vv_k, vv'$.

\textbf{Case 2.1:} Colors $2_a,2_b$ are not used on $u_1v_1, v_1v_k, v_ku_k$. Then $f(u_1v_1), f(v_1v_k), f(v_ku_k)$ can be $1_a, 3, 1_a$, or $1_a, 3, 1_b$, or $3, 1_a, 1_b$ up to symmetry. If it is $1_a, 3, 1_a$, then we can always recolor $v_1v_k$ with $1_b$ (for all possibilities of $f(vv_1), f(vv_k),$ $ f(vv')$), which was solved in \textbf{Case 1}. Therefore, $f(u_1v_1), f(v_1v_k), f(v_ku_k)$ is $1_a, 3, 1_b$, or $3, 1_a, 1_b$. In the former case, we can always recolor $f(vv_1), f(vv_k), f(vv')$ to $1_b, 1_a, 2_a$. We color $v_{k-1}v_k, v_{k-1}v_{k-1}', v_{k-2}v_{k-1}$ with $1_a, 1_b, 2_a$, $v_1v_2$ until $v_{k-3}v_{k-2}$ together with $v_{k-2}v_{k-2}'$ with $1_b, 1_a, \ldots$, and (if $k \ge 5$) $v_{k-3}v_{k-3}'$ until $v_1v_1'$ with $2_b, 2_a, \ldots$. In the latter case, we color $v_{k-1}v_k$ with $2_a$, $v_1v_2$ until $v_{k-2}v_{k-1}$ together with $v_{k-1}v_{k-1}'$ with $1_b, 1_a, \ldots$, and $v_{k-2}v_{k-2}'$ until $v_1v_1'$ with $2_b, 2_a, \ldots$.

\textbf{Case 2.2:} Colors $2_a,2_b$ are used once on $u_1v_1, v_1v_k, v_ku_k$. Then $f(u_1v_1), f(v_1v_k), f(v_ku_k)$ can be $1_a, 3, 2_a$, or $1_a, 2_a, 3$, or $2_a, 1_a, 3$ up to symmetry. If it is $1_a, 3, 2_a$ or $1_a, 2_a, 3$, then we can always recolor $v_1v_k$ with $1_b$ (for all possibilities of $f(vv_1), f(vv_k),$ $ f(vv')$), which was solved in \textbf{Case 1} or \textbf{Case 2.1}. Therefore, $f(u_1v_1), f(v_1v_k), f(v_ku_k)$ are colored with  $2_a, 1_a, 3$ and $f(vv_1), f(vv_k),$ $ f(vv')$ can only be colored with $1_b, 2_b, 1_a$ or $2_b, 1_b, 1_a$. In the former case, we color $v_1v_2,v_1v_1',v_2v_3$ with $1_b,1_a,2_b$, $v_{k-1}v_k$ until $v_3v_4$ together with $v_3v_3'$ with $1_b,1_a,\ldots$, and (if $k \ge 5$) $v_4v_4'$ until $v_{k-1}v_{k-1}'$ with $2_a,2_b,\ldots$. In the latter case, we color $v_1v_2$ with $2_b$, $v_{k-1}v_k$ until $v_2v_3$ together with $v_2v_2'$ with $1_b,1_a,\ldots$, and $v_3v_3'$ until $v_{k-1}v_{k-1}'$ with $2_a, 2_b, \ldots$. 
\end{proof}

Let $P=F_0F_{\ell} \ldots F_2F_1$ be a longest path in the weak dual of $B_1$ starting at $F_0$. We first show $\ell \ge 2$.
%Then we describe the ``outermost two layers'' of the block $B_1$.

\begin{lemma}\label{11223-two-layers}
$\ell \ge 2$.   
\end{lemma}

\begin{proof}
Suppose not, i.e., $\ell = 1$ ($F_2 = F_0$) and $F_2$ has boundary cycle $u_0u_1 \ldots u_k$, $k \ge 2$. By Lemma~\ref{11223-outside-1}, every pendant face except $F_2$ is a $3$-face. Since each of $u_i$ is a $3$-vertex, let the other neighbour not in $F_2$ to be $u_i'$, $1 \le i \le k$. If there is a $3$-face sitting on top of $u_i,u_{i+1}$, then $u_i' = u_{i+1}'$ and let the leaf attached on it be called $u_i''$. We delete $u_1, u_1', \ldots, u_k, u_k'$ and all $u_i''$ (if exists) to obtain a graph $G'$. By the minimality of $G$, $G'$ has a good coloring $f$. By symmetry, we may assume $f(v_0u_0)$ is either $1_a,2_a,$ or $3$. Note that when $k = 3$ we may assume the only pendant face is on top of $u_1,u_2$. 

Let $F$ be a pendant face adjacent to $F_2$ with $V(F) \cap V(F_2) = \{u_i,u_{i+1}\}$, $1 \le i \le k-1$. We may now replace $F$ with a $4$-face $u_iu_i'u_{i+1}'u_{i+1}u_i$. Furthermore, we may add vertices $u_i'',w_i$ with edges $u_i''u_i', u_i''u_{i+1}', u_i''w_i$ to form an HC graph $H_i$. In this way, we are coloring a larger graph than $G'$ and can use the proof in the future. At the same time, we can always use the colors of $u_iu_i', u_{i+1}u_{i+1}', u_i'u_{i+1}'$ of $H_i$ to color the edges $u_iu_i', u_{i+1}u_{i+1}', u_i'u_i''$ of $G$.
%if $k = 2$, then we color $u_1u_2, u_1u_1', u_1'u_2,u_1'u_1''$ with $1_a, 2_b, 1_b, 1_a$. By symmetry, we can either color $u_0u_k$ with $2_a$ or $3$. In the former case,

\textbf{Case 1:} $f(v_0u_0) = 1_a$. We can always color $u_0u_1$ with $1_b$. We can color $u_0u_k$ with $2_a$ or $2_b$ unless $\{f(v_0v_0'), f(v_0v_0'')\} = \{2_a, 2_b\}$. In this case, each of $v_0'$ and $v_0''$ must see $1_a$ and $1_b$, and hence we can color $u_0u_k$ with $3$. 

%The whole block do not need to be Type I,II,III,IV, V
\textbf{Case 1.1:} We can color $u_0u_k$ with $2_a$. In case $k = 3$, we may assume there is an HC graph $H_1$ on top of $u_1,u_2$. We color $u_1u_2, u_2u_3, u_3u_3'$ with $2_b, 1_a, 1_b$, and $H_1$ with Type I coloring. Otherwise, $k \neq 3$. We color $u_1u_2, \ldots, u_{k-1}u_k$ together with $u_ku_k'$ with $1_a, 1_b, \ldots$, (if $k \ge 4$) $u_1u_1', \ldots, u_{k-2}u_{k-2}'$ with $2_b, 2_a, \ldots$. If there is an HC graph $H_{k-2}$ on top of $u_{k-2}, u_{k-1}$, then we recolor $u_{k-1}u_k$ with $2_b$, $H_{k-2}$ with Type IV coloring, and other HC graphs with Type II coloring. If there is no HC graph on top of $u_{k-2}, u_{k-1}$, color $u_{k-1}u_{k-1}'$ with $3$, and all HC graphs with Type II coloring (except use Type IV coloring on the HC graph on top of $u_{k-1}, u_k$ if it exists).

\textbf{Case 1.2:} We can color $u_0u_k$ with $3$. We color $u_1u_2, \ldots, u_{k-1}u_k$ together with $u_ku_k'$ with $1_a, 1_b, \ldots$, $u_1u_1', \ldots, u_{k-1}u_{k-1}'$ with $2_a, 2_b, \ldots$. If there is an HC graph $H_{k-1}$ on top of $u_{k-1}, u_k$, then we switch the colors of $u_{k-1}u_k$ and $u_{k-1}u_{k-1}'$ and can color $H_{k-1}$ with Type V coloring. We use type II coloring for other HC graphs. Note that if $k = 3$ we assume by symmetry that the HC graph is on top of $u_2,u_3$ and can use the aforementioned coloring procedure.

\textbf{Case 2:} $f(v_0u_0) = 2_a$ or $3$. We color $u_0u_1, u_0u_k$ with $1_a, 1_b$. If $k$ is odd, color $u_1u_2, \ldots, u_{k-1}u_k$ with $1_b, 1_a, \ldots, 1_a$, $u_1u_1', \ldots, u_ku_k'$ with $2_b, 2_a, \ldots, 2_b$, all HC graphs with Type II coloring. If $k$ is even, then color $u_1u_1', u_1u_2,$ $ u_2u_2'$ with $1_b, 2_b, 1_a$, and (if $k \ge 4$) $u_2u_3, \ldots, u_{k-1}u_k$ with $1_b, 1_a, \ldots, 1_a$, $u_3u_3', \ldots, u_ku_k'$ with $2_a, 2_b, \ldots, 2_b$. If there is an HC graph $H_1$ ($H_2$) on top of $u_1, u_2$ ($u_2, u_3$), then use Type I (Type III) coloring. We use type II coloring for other HC graphs.
%We know $v_0$ sees $1_a, 1_b$, since otherwise we can recolor $v_0u_0$ with $1_a$ or $1_b$, which was solved in \textbf{Case 1}.
%If $k = 2$, then we color $u_0u_1, u_1u_2, u_0u_2,u_1u_1', u_1'u_2,u_1'u_1''$ with $1_a, 2_b, 1_b, 1_b, 1_a, 2_a$. If $k = 3$, then we color $u_0u_1, u_1u_2, u_2u_3, u_3u_0, u_1u_1', u_2u_1', u_1'u_1'', u_3u_3'$ with $1_b, 2_b, 3, 1_a, 1_a, 1_b, 2_a, 1_b$, which is Type IV. If $k \ge 4$, then we color $u_0u_k, u_1u_1'$ with $2_b, 3$, $u_0u_1$ until $u_{k-1}u_k$ together with $u_ku_k'$ with $1_a, 1_b, \ldots$, $u_{k-1}u_{k-1}'$ until $u_2u_2'$ with $2_a, 2_b, \ldots$. Note that each possible $u_i'u_i''$ has at least one available color from $\{1_a, 1_b\}$ and thus can be colored.
%\textbf{Case 3:} $f(v_0u_0) = 3$. Exactly the same with \textbf{Case 2}.%, $v_0$ sees $1_a, 1_b$. If $k = 2$, then we use the same coloring as \textbf{Case 2}. If $k = 3$, then we color $u_0u_1, u_1u_2, u_2u_3, u_3u_0, u_1u_1', u_2u_1', u_1'u_1'', u_3u_3'$ with $1_a, 2_b, 1_b, 2_a, 1_b, 1_a, 2_a, 1_a$, which is Type V. If $k \ge 4$ and $k$ is odd, then we color $u_0u_1, \ldots, u_ku_0$ with $1_a, 1_b, \ldots, 1_a, 1_b$, and $u_1u_1', \ldots, u_ku_k'$ with $2_a, 2_b, \ldots, 2_a$. If $k \ge 4$ and $k$ is even, then we color $u_0u_k$ with $2_a$, $u_0u_1$ until $u_{k-1}u_k$ together with $u_ku_k'$ with $1_a, 1_b, \ldots, 1_a$, $u_{k-1}u_{k-1}', \ldots, u_1u_1'$ with $2_b, 2_a, \ldots, 2_b$. Note that each possible $u_i'u_i''$ has at least one available color from $\{1_a, 1_b\}$ and thus can be colored.
\end{proof}

Then we show $F_2$ must be a $4$-face when $\ell \ge 2$.

\begin{lemma}\label{11223-outside-2}
$F_2$ is a $4$-face. 
\end{lemma}

\begin{proof}
By Lemma~\ref{11223-two-layers}, $\ell \ge 2$. Suppose $F_2$ is a $k$-face, where $k \ge 5$. Let the boundary cycle of $F_2$ be $v_1 \ldots v_k$ with $V(F_3) \cap V(F_2) = \{v_1, v_k\}$. We delete $v_2, \ldots, v_{k-1}$ and all pendant faces hang on $F_2$ with all their leaves, and add vertices $z,z_1,z_2,z_k$ with edges $z_1v_1,z_kv_k,z_1z_k,z_1z_2,z_2z_k,z_2z$ (HC graph $H$) to obtain a graph $G'$. By the minimality of $G$, $G'$ has a good coloring $f$. We extend $f$ to $G'$. Note that each $v_i$ with no pendant faces on top has a leaf neighbour $v_i'$, $2 \le i \le k-1$. Let $F$ be a pendant face adjacent to $F_2$ with $V(F) \cap V(F_2) = \{v_i,v_{i+1}\}$, $2 \le i \le k-1$. We may now replace $F$ with a $4$-face $v_iv_i'v_{i+1}'v_{i+1}v_i$ and add vertices $v_i'',w_i$ with edges $v_i''v_i', v_i''v_{i+1}', v_i''w_i$ to form an HC graph $H_i$. In this way, we are coloring a larger graph than $G'$ and can use the proof in the future. At the same time, we can always use the colors of $v_iv_i', v_{i+1}v_{i+1}', v_i'v_{i+1}'$ of $H_i$ to color the edges $v_iv_i', v_{i+1}v_{i+1}', v_i'v_i''$ of $G$. Let $N(v_1) = \{u_1, v_2, v_k\}$ and $N(v_k) = \{u_k,v_1, v_{k-1}\}$. 

\textbf{Case 1:} The HC graph $H$ uses type I coloring, i.e., $u_1v_1,v_1v_k,v_ku_k,v_1z_1,v_kz_k,z_1z_k,z_1z_2,z_2z_k,z_2z$ are colored with $1_a, 2_a, 1_b, 1_b, 1_a, 2_b, 1_a, 1_b, 2_a$. If $k$ is odd, then we color $v_1v_2$ until $v_{k-1}v_{k}$ with $1_b, 1_a, \ldots, 1_b, 1_a$, $v_2v_2'$ until $v_{k-1}v_{k-1}'$ with $2_b, 2_a, \ldots, 2_b$. If there is an HC graph on top of $v_i, v_{i+1}$, $2 \le i \le k-2$, then we use Type II coloring to complete the extension. Let $j$, $2 \le j \le k-2$, be the smallest index such that there is an HC graph $H_j$ on top of $v_jv_{j+1}$. If $k$ is even, then $k \ge 6$, we color $v_1v_2, \ldots, v_{j-1}v_j, v_{j+1}v_{j+2}, \ldots, v_{k-1}v_k$ with $1_b, 1_a, \ldots, 1_b, 1_a$, $v_2v_3, v_4v_4', \ldots, v_{k-1}v_{k-1}'$ with $2_b, 2_a, \ldots, 2_b$ when $j = 2$ and $v_2v_2', \ldots, v_{j-1}v_{j-1}', v_jv_{j+1}, v_{j+2}v_{j+2}', \ldots, v_{k-1}v_{k-1}'$ with $2_b, 2_a, \ldots,$ $ 2_b$ when $3 \le j \le k-3$ and $v_2v_2', \ldots, v_{j-1}v_{j-1}', v_jv_{j+1}$ with $2_b, 2_a, \ldots, 2_b$ when $j = k-2$, and color $H_j$ with Type I coloring, and the remaining HC graphs with Type II coloring.

\textbf{Case 2:} $H$ uses type II coloring, i.e., $u_1v_1,v_1v_k,v_ku_k,v_1z_1,v_kz_k,z_1z_k,z_1z_2,z_2z_k,z_2z$ are colored with $1_a, 1_b, 1_a, 2_a, 2_b,$ $1_b, 1_a, 3, 1_b$. We color $v_1v_2, v_{k-1}v_k$ with $2_a,2_b$. If $k = 5$, then we may assume there is an HC graph $H_2$ on top of $v_2,v_3$, color $v_2v_3,v_3v_4,v_4v_4'$ with $1_b, 1_a, 1_b$, and $H_2$ with type IV coloring. If $k \ge 6$ and we color $v_2v_2', v_2v_3, \ldots, v_{k-2}v_{k-1}$ together with $v_{k-1}v_{k-1}'$ with $1_a, 1_b, \ldots$. When there is no HC graph on top of $v_3,v_4$, we color $v_3v_3'$ with $3$ and $v_{k-2}v_{k-2}', \ldots, v_4v_4'$ with $2_a, 2_b, \ldots$. We use type IV coloring if there is an HC graph on $v_2,v_3$. When there is an HC graph $H_3$ on top of $v_3,v_4$, we color $v_3v_3',v_4v_4'$ with $1_b,3$, recolor $v_2v_3$ with $2_b$, and (if $k \ge 7$) color $v_{k-2}v_{k-2}', \ldots, v_5v_5'$ with $2_a, 2_b, \ldots$, and use type IV coloring on $H_3$. We use Type III coloring if there is an HC graph on $v_{k-2},v_{k-1}$. Since at most one of $v_iv_i'$ uses color $3$, we use Type II coloring for other HC graphs in the way that all pairs of color $3$s are at distance at least $4$. 

\textbf{Case 3:} $H$ uses type III coloring, i.e., $u_1v_1,v_1v_k,v_ku_k,v_1z_1,v_kz_k,z_1z_k,z_1z_2,z_2z_k,z_2z$ are colored with $2_a, 1_a, 1_b, 1_b, 2_b,$ $ 1_a, 3, 1_b, 1_a$. We color $v_1v_2, v_{k-1}v_k$ with $1_b, 2_b$. If $k$ is odd, we color $v_2v_3, \ldots, v_{k-2}v_{k-1}$ together with $v_{k-1}v_{k-1}'$ with $1_a, 1_b, \ldots, 1_a$, $v_2v_2', \ldots, v_{k-2}v_{k-2}'$ with $2_b, 2_a, \ldots, 2_a$. We use Type III coloring for the HC graph on $v_{k-2}, v_{k-1}$ (if there is one) and Type II coloring on other HC graphs. If $k$ is even, then $k \ge 6$, we color $v_2v_2', v_2v_3, v_3v_3'$ with $1_a, 2_b, 1_b$, $v_3v_4, \ldots, v_{k-2}v_{k-1}$ together with $v_{k-1}v_{k-1}'$ with $1_a, 1_b, \ldots, 1_a$, $v_4v_4', \ldots, v_{k-2}v_{k-2}'$ with $2_a, 2_b, \ldots, 2_a$. We use Type I coloring for the HC graph on $v_2,v_3$ (if there is one), Type III coloring for the HC graph on $v_{k-2}, v_{k-1}$ (if there is one), and Type II coloring for other HC graphs.

\textbf{Case 4:} $H$ uses type IV coloring, i.e., $u_1v_1,v_1v_k,v_ku_k,v_1z_1,v_kz_k,z_1z_k,z_1z_2,z_2z_k,z_2z$ are colored with $2_a, 1_a, 1_b, 1_b, 3, $ $2_b, 1_a, 1_b, 2_a$. We color $v_1v_2, v_{k-1}v_k$ with $1_b,3$. We color $v_2v_3, \ldots, v_{k-2}v_{k-1}$ together with $v_{k-1}v_{k-1}'$ with $1_a, 1_b, \ldots$, $v_2v_2', \ldots, v_{k-2}v_{k-2}'$ with $2_b, 2_a, \ldots$. If there is an HC graph $H_{k-2}$ on top of $v_{k-2},v_{k-1}$, we switch the colors of $v_{k-2}v_{k-1}$ and $v_{k-2}v_{k-2}'$, and use Type V coloring on $H_{k-2}$. We use Type II coloring for other HC graphs in the way that all pairs of color $3$s are at distance at least $4$.

\textbf{Case 5:} $H$ uses type V coloring, i.e., $u_1v_1,v_1v_k,v_ku_k,v_1z_1,v_kz_k,z_1z_k,z_1z_2,z_2z_k,z_2z$ are colored with $3, 2_a, 1_a, 1_a, 1_b,$ $2_b, 1_b, 1_a, 2_a$. We color $v_{k-1}v_k$ with $1_b$. If $k$ is odd, we color $v_{k-2}v_{k-1}, \ldots, v_1v_2$ with $1_a, 1_b, \ldots, 1_a$, $v_2v_2', \ldots, v_{k-1}v_{k-1}'$ with $2_b, 2_a, \ldots, 2_b$. We color all HC graphs with Type II coloring. If $k$ is even, we color $v_{k-1}v_{k-1}', v_{k-2}v_{k-1}, v_{k-2}v_{k-2}'$ with $1_a, 2_b, 1_b$, $v_{k-3}v_{k-2}, \ldots, v_1v_2$ with $1_a, 1_b, \ldots, 1_a$, $v_2v_2', \ldots, v_{k-3}v_{k-3}'$ with $2_b, 2_a, \ldots, 2_a$. If there is an HC graph $H_{k-2}$ ($H_{k-3}$) on top of $v_{k-2},v_{k-1}$ ($v_{k-3},v_{k-2}$), then we color $H_{k-2}$ ($H_{k-3}$) with Type I (Type III) coloring. We use Type II coloring for other HC graphs.
\end{proof}

Now, we complete the proof. By Lemma~\ref{11223-two-layers}, $\ell \ge 2$. We split the proof into two cases.

\textbf{Case i:} $\ell = 2$. By Lemma~\ref{11223-outside-2}, $B_1$ can be described as $F_0$ with a few leaf edges, $3$-faces with a leaf edge, and HC graphs attached on it. We can use the same argument used in the proof of Lemma~\ref{11223-two-layers}, since we treated every $3$-face with a leaf edge as an HC graph in the coloring procedure. In case $B_1$ is the only non-trivial block and $F_0$ is a $3$-face, by Lemma~\ref{11223-outside-1} and~\ref{11223-outside-2}, $F_2$ is a $4$-face and $F_1$ is a $3$-face. We color both HC graphs with Type I coloring.

% ADD A PICTURE
\textbf{Case ii:} $\ell \ge 3$. Let $V(F_3) \cap V(F_4) = \{v_1, v_k\}$. We delete $V(F_3) - \{v_1, v_k\}$, all leaf edges, all $3$-faces with a leaf edge, all HC graphs attached on $F_3$, and add vertices $z,z_1,z_2,z_k$ with edges $z_1v_1,z_kv_k,z_1z_k,z_1z_2,z_2z_k,z_2z$ (HC graph) to obtain a graph $G'$. By the minimality of $G$, $G'$ has a good coloring $f$. We can use the argument used in the proof of Lemma~\ref{11223-outside-2}, since we treated every $3$-face with a leaf edge as an HC graph in the color-extending procedure. \hfill \qed

\subsection{$(1^2,2^2,k_2')$-packing edge-coloring of $2$-connected subcubic outerplanar graphs}\label{1122k'}

We provide a $2$-connected subcubic outerplanar graph that is not $(1^2,2^2,12)$-colorable, showing $k_2' \le 11$.

\begin{example}
Let $G_2$ and $G_3$ be the graph shown in Fig.~\ref{example_7}. We show $G_3$ is a $2$-connected subcubic outerplanar graph that is not $(1^2,2^2,12)$-packing edge-colorable. 
\end{example}

\begin{figure}[ht]
\begin{center}
 \vspace{-9mm}
\includegraphics[scale=0.5]{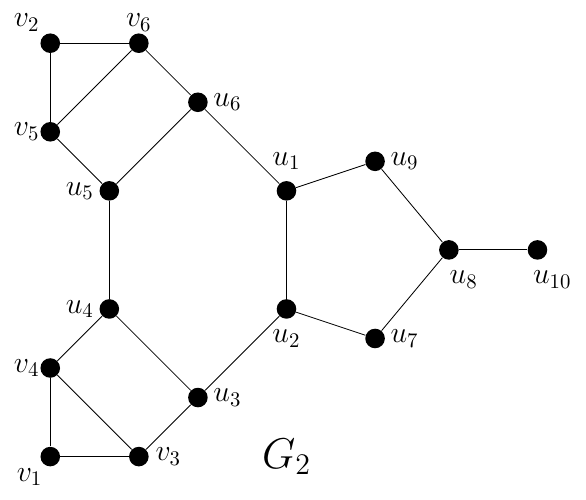} \hspace{10mm}
\includegraphics[scale=0.45]{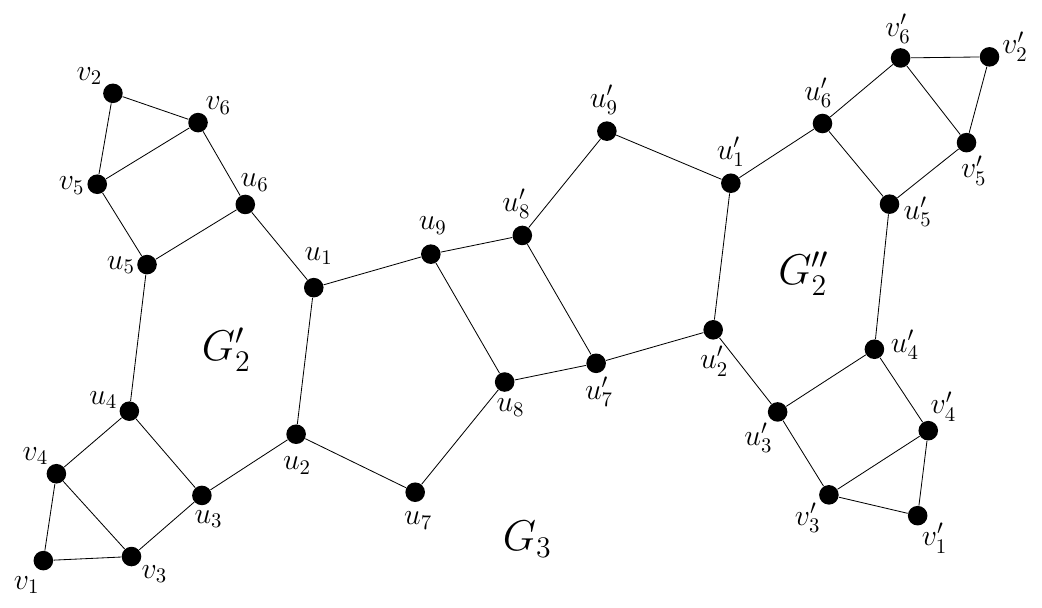}% 
\vspace{-3mm}
\caption{A $2$-connected subcubic outerplanar graph that is not $(1^2,2^2,12)$-packing edge-colorable.}\label{example_7}
\end{center}
\vspace{-8mm}
\end{figure}

\begin{proof}
We first show $G_2$ has no $(1^2,2^2)$-coloring. Suppose $G_2$ has a $(1^2,2^2)$-coloring $f$. By Claim~\ref{house1122}, $G_0$ in Example~\ref{ex1} has only two $(1^2,2^2)$-coloring up to symmetry and the two copies of $G_0$ in $G_2$ must both be of Type II. We may assume $f(u_2u_3) = 1_a, f(u_3u_4) = 2_a, f(u_4u_5) = 1_b, f(u_5u_6) = 2_b, f(u_1u_6) = 1_a$. Therefore, $f(u_1u_2) = 1_b, f(u_1u_9) = 2_a, f(u_2u_7) = 2_b$ and $\{f(u_8u_9), f(u_7u_8)\} = \{1_a, 1_b\}$. However, we cannot color $u_8u_{10}$, which is a contradiction.

Suppose $G_3$ has a $(1^2,2^2,12)$-coloring. Note that $G_3$ has two disjoint copies of $G_2$, i.e., $G_2'$ and $G_2''$ in Fig.~\ref{example_7}. Therefore, each of $G_2'$ and $G_2''$ must use color $12$ at least once. However, the longest possible distance from an edge in $G_2'$ to $G_2''$ is $12$ (realized by $v_1v_4$ and $v_2'v_5'$), which is a contradiction.
\end{proof}

\vspace{-3mm}
\section{Open Problems}

We have shown in this paper that $3 \le k_1 \le 6$, $k_1' = 2$, $3 \le k_2 \le 4$, and $3 \le k_2' \le 11$. We have a proof showing $k_2' \ge 4$. However, we decided to not include it in this paper since the argument is similar to the proof of Theorem~\ref{11223-main-theorem}. We tend to believe that $k_2 = 4$ and applying stronger conditions to the $(1^2,2^2,4)$-coloring may lead to a proof. It would be nice if the remaining three numbers $k_1,k_2,k_2'$ can be determined exactly. We post those questions as open problems. 

\begin{itemize}
    \item What is the largest $k_1$ such that every subcubic outerplanar graph is $(1,2^4,k_1)$-colorable?

    \item What is the largest $k_2$ such that every subcubic outerplanar graph is $(1^2,2^2,k_2)$-colorable?

    \item What is the largest $k_2'$ such that every $2$-connected subcubic outerplanar graph is $(1^2,2^2,k_2')$-colorable?
\end{itemize}

%%%%%%%%%%%%%%%%%%%%%%%%%%%%%%%%%%%%%%%%%%%%%%%%%%%%%%%%%%%%%%%%%%%%%%%%%%
\vspace{-5mm}
%\vspace{-5mm}

\end{document}